\documentclass{amsart}

\usepackage{amssymb}
\usepackage{color}
\usepackage{amsxtra} 
\usepackage{mathrsfs} 
\usepackage{yfonts}

\numberwithin{equation}{section}

\newtheorem{thm}{Theorem}[section]
\newtheorem{prop}[thm]{Proposition}
\newtheorem{lem}[thm]{Lemma}

\theoremstyle{definition}

\theoremstyle{remark}
\newtheorem{rem}[thm]{Remark}

\allowdisplaybreaks[3]

\newcommand{\al}{\alpha}
\newcommand{\be}{\beta}
\newcommand{\ga}{\gamma}

\newcommand{\de}{\delta}

\newcommand{\e}{\varepsilon}

\newcommand{\la}{\lambda}
\newcommand{\La}{\Lambda}
\newcommand{\om}{\omega}
\newcommand{\Om}{\Omega}
\newcommand{\sgm}{\sigma}
\renewcommand{\th}{\theta}

\newcommand{\p}{\partial}
\newcommand{\I}{\infty}
\newcommand{\Sc}[1]{\mathcal{#1}}
\newcommand{\F}{\Sc{F}}
\newcommand{\cL}{\Sc{L}}

\newcommand{\Bo}[1]{\mathbb{#1}}
\newcommand{\R}{\Bo{R}}
\newcommand{\T}{\Bo{T}}
\newcommand{\HP}{\Bo{P}}
\newcommand{\osc}{_{\mathrm{osc}}}

\newcommand{\lec}{\lesssim}
\newcommand{\gec}{\gtrsim}
\newcommand{\hhat}{\widehat}
\newcommand{\bbar}{\overline}
\newcommand{\ti}{\widetilde}

\newcommand{\shugo}[1]{\{ #1\}}
\newcommand{\Shugo}[2]{\big\{ \, #1 \, \big| \, #2 \, \big\}}

\newcommand{\LR}[1]{{\langle #1 \rangle }}

\newcommand{\norm}[2]{\big\| #1 \big\| _{#2}}
\newcommand{\tnorm}[2]{\| #1 \| _{#2}}

\newcommand{\eq}[2]{\begin{equation} \label{#1} \begin{split} #2 \end{split} \end{equation}}
\newcommand{\IVP}[2]{\eq{#1}{\left\{ \begin{split} #2 \end{split} \right.}}
\newcommand{\eqq}[1]{\begin{align*} #1 \end{align*}}
\newcommand{\eqs}[1]{\begin{gather*} #1 \end{gather*}}
\newcommand{\mat}[1]{\begin{smallmatrix} #1 \end{smallmatrix}}

\newcommand{\hx}{\hspace{10pt}}
\newcommand{\hxx}{\hspace{30pt}}

\begin{document}

\title[The rotating Navier-Stokes eq. with fractional Laplacian]
{Global solvability of the rotating Navier-Stokes equations with fractional Laplacian in a periodic domain} 

\author{Nobu Kishimoto}
\address{
Research Institute for Mathematical Sciences, Kyoto University, 
Kitashirakawa-Oiwakecho, Sakyo, Kyoto 606-8502, Japan
} 
\email{nobu@kurims.kyoto-u.ac.jp} 

\author{Tsuyoshi Yoneda}
\address{
Graduate School of Mathematical Sciences, University of Tokyo,
Komaba 3-8-1 Meguro, Tokyo 153-8914, Japan
}
\email{yoneda@ms.u-tokyo.ac.jp}

\date{\today}

\begin{abstract}
We consider existence of global solutions to equations for three-dimensional rotating fluids in a periodic frame provided by a sufficiently large Coriolis force.
  The Coriolis force appears in almost all of the
models of  meteorology and geophysics dealing with large-scale
phenomena. 
In the spatially decaying case,  Koh, Lee and Takada (2014) showed existence for the large times of solutions of the rotating Euler equations
provided by the large Coriolis force. 
In this case the resonant equation does not appear anymore.
In the periodic case, however, the resonant equation appears, 
and thus the main subject in this case is to show existence of global solutions to the resonant equation.
Research in this direction was initiated by Babin, Mahalov and Nicolaenko (1999) who treated the rotating Navier-Stokes equations on general periodic domains.
On the other hand, Golse, Mahalov and Nicolaenko (2008) considered bursting dynamics of the resonant equation in the case of a cylinder with no viscosity. Thus we may not expect to show global existence of solutions to the resonant equation without viscosity in the periodic case.
In this paper we show existence of global solutions for fractional Laplacian case (with its power strictly less than the usual Laplacian) in the periodic domain with the same period in each direction.
The main ingredient is an improved estimate on resonant three-wave interactions, which is based on a combinatorial argument.
\end{abstract}

\keywords{Navier-Stokes equations, Coriolis force, global regularity, resonant equation, divisor bound}

\subjclass[2010]{76U05,42B05}

\maketitle


\section{Introduction}

We consider the rotating three-dimensional Navier-Stokes equations with the fractional Laplacian:

\begin{eqnarray}\label{NSC}
& &\partial_t u +(u \cdot \nabla) u+ \Omega e_3 \times u +\nu (-\Delta)^\alpha u =- \nabla
p
\quad\text{in}\quad
\mathbb{T}^3:=[0,2\pi)^3
,\\
& &
 \nabla \cdot u=0\quad
 and
 \quad
  \ u|_{t=0}=u_0,
\nonumber
\end{eqnarray}
where $u=u(t)=(u^{1}(t,x),u^{2}(t,x), u^{3}(t,x))$ is the unknown
velocity vector field and $p=p(t,x)$ is the unknown scalar pressure
at the point $x=(x_1,x_2,x_3)\in [0,2\pi)^3$ in space and time $t>0$
while $u_0=u_0(x)$ is the given initial velocity field. Here $\Omega
\in \mathbb{R}$ is the Coriolis parameter, which is twice the angular
velocity of the rotation around the vertical unit vector
$e_3=(0,0,1)$, and $\nu >0$ is the kinematic viscosity coefficient. 
By $\times$ we denote the exterior product, and hence, the
Coriolis term is represented by%
\footnote{
Vectors in $\mathbb{R}^3$ should be considered as \emph{column vectors}, but we will write them as \emph{row vectors} throughout the paper for notational convenience.
}
$e_3 \times u= Ju$ with the
corresponding skew-symmetric $3 \times 3$ matrix $J$, namely,
\eqq{J:=\left( \begin{matrix} 0&-1&0\\ 1&0&0\\ 0&0&0\end{matrix}\right) .}
Note that $J$ restricted to divergence free vector fields is in fact a non-local zero order pseudo-differential operator.

The Coriolis force plays a significant role in the large scale flows
considered in meteorology and geophysics. In 1868 Kelvin observed
that a sphere moving along the axis of uniformly rotating water
takes with it a column of liquid as if this were a rigid mass (see
\cite{F} for references). After that, Taylor \cite{Ta} and Proudman
\cite{P} did important contributions.
Mathematically, linear wave dynamics for rotating fluids was
investigated by Poincar\'e \cite{Po}, more recently, by Babin, Mahalov and Nicolaenko
\cite{BMN1,BMN2} using the fully Navier-Stokes equations in a periodic domain.

Throughout this paper we essentially use the spatial Fourier transform denoted by $\F$ or $\hhat{\;\cdot\;}$:
\eqq{u(x)=\sum _{n\in \Bo{Z}^3}\hhat{u}(n)e^{in\cdot x}
\quad\text{with}\quad(\F u)(n)= \hhat{u}(n):=\frac{1}{(2\pi )^{3}}\int _{\T ^3}u(x)e^{-in\cdot x}\,dx.}
Let us define the inhomogeneous Sobolev  spaces $H^s$ as follows:
\begin{equation*}
H^s(\mathbb{T}^d):=\bigg\{ u=\sum_{n\in\mathbb{Z}^d}\hat u(n)e^{in\cdot x}\,\bigg| \,\|u\|_{H^s}:=\Big( \sum_{n\in\mathbb{Z}^d}(1+|n|^2)^{s}|\hat u (n)|^2\Big) ^{1/2}<\infty \bigg\}.
\end{equation*}
The homogeneous version $\dot H^s$ can be defined as 
\begin{equation*}
\dot H^s(\mathbb{T}^d):=\bigg\{ u=\sum_{n\in\mathbb{Z}^d}\hat u(n)e^{in\cdot x}\,\bigg|\, \|u\|_{\dot H^s}:=\Big( \sum_{n\in\mathbb{Z}^d}|n|^{2s}|\hat u (n)|^2\Big) ^{1/2}<\infty \bigg\}.
\end{equation*}

We will assume that all the vector fields in this paper are mean-zero.
This assumption is valid from the following observation:
Let
\eqq{f(t):=\left( \begin{matrix} \hhat{u}_{0}^1(0)\cos \Om t+\hhat{u}^2_0(0)\sin \Om t, & -\hhat{u}_0^1(0) \sin \Om t +\hhat{u}_0^2(0)\cos \Om t, & \hhat{u}_0^3(0)\end{matrix}\right) .}
Note that $f(t)$, which is the solution to the following ODE:
\begin{equation*}
f'(t) + \Omega Jf(t) = 0,\qquad f(0)=\hhat{u}_0(0),
\end{equation*}
is the average over $\T^3$ of the velocity component of the solution to (\ref{NSC}) at $t$.
Then the following invertible transforms
\begin{equation*}
 u(t,x) \mapsto u\left(t,x+\int _0^t f(s)ds\right) -f(t)\quad\text{and}\quad p(t,x) \mapsto p\left(t,x+\int _0^t f(s)ds\right)
\end{equation*}
preserve the equation \eqref{NSC}, and the new velocity field has zero mean for all time. 
We therefore do not distinguish homogeneous and inhomogeneous Sobolev spaces.

Let us recall the result of Babin \mbox{et al.} as the starting point of our work:
\begin{thm}[\cite{BMN2}]\label{AA}
Let $s>1/2$ and $\mathbf{T}^3$ be a torus with \emph{arbitrary} period (distinguished from $\mathbb{T}^3$). 
Let $u_0\in H^s(\mathbf{T}^3)$ be a divergence-free vector field. Then there exists a positive
$\Omega_0$ depending on $\| u_0\| _{H^s}$ and the period of torus 
such that
for all $|\Omega|\geq \Omega_0$, there is a unique global solution
$$
u(t)\in C([0,\infty);H^{s}(\mathbf{T}^3))\cap L^2((0,\infty );H^{s+1}(\mathbf{T}^3))
$$
to the equation (\ref{NSC}) with $\alpha=1$.
\end{thm}
As shown in \cite{BMN1,BMN2}, it turns out that the estimates on the obtained global solutions depend crucially on the period of the torus.
For instance, the global a priori bound is independent of the viscosity coefficient $\nu >0$ for generic periods (\cite{BMN1}), whereas exponential-in-$\nu ^{-1}$ dependence may occur in the ``worst case'' (\cite{BMN2}).
In this paper, we will focus on the special torus $\mathbb{T}^3=[0,2\pi )^3$, which is among the ``worst case'' as we will see in Section~\ref{subsec:BMN} below.
We remark that the above result was extended to the critical case $s=1/2$ in \cite[Theorem 6.2]{CDGG}; see also \cite[Theorem~5.7]{CDGG} for an analogous result on $\mathbb{R}^3$.

We next recall previous results in the inviscid case.
In the spatially decaying setting, combining the Strichartz estimates with Beale-Kato-Majda's blow-up criterion, Koh, Lee and Takada \cite{KLT} showed long time existence of solutions 
to the Euler equations provided by large Coriolis parameter. 
The periodic case may be more difficult due to the appearance of the resonant equation.
In \cite{BMN1}, Babin \mbox{et al.} initially considered long time solvability of the rotating Euler equations (see also \cite{MNBG} in a cylinder case). However they set  specific periodic domains (specific aspect ratios) and  eliminate ``nontrivial resonant part''
which is essentially related to the Rossby wave in physics (see \cite{KY, YY} for example).
For domains with other periods we need to deal with ``nontrivial resonant part'', and it has been an open problem.
On the other hand, in a cylinder case, Golse, Mahalov and Nicolaenko \cite{GoMN} considered bursting dynamics of the inviscid resonant equation. Thus we may not expect to show existence of inviscid  smooth global flow in general periodic cases.
Nevertheless, by a refined estimate on ``nontrivial resonant part'' based on elementary number theory (Lemma \ref{lem:c2} below),
we can  progress a less viscosity effect case (fractional Laplacian case) in the periodic domain $\mathbb{T}^3=[0,2\pi )^3$.
A fractional Laplacian has been employed in many theoretical and numerical works instead of the usual viscosity; see, for example, \cite{CS} and \cite{XWCE}.

First we state the following local existence theorem, which is obtained by a standard argument.
For $\al \in (\frac{3}{4},1]$, $C(\al )$ denotes any positive constant depending on $\al$ with $C(\al )\to \I$ as $\al \downarrow \frac{3}{4}$.
\begin{thm}\label{prop:lwp}
Let $\nu >0$, $\al \in (\frac{3}{4},1]$, $s\ge 1$.
Then, \eqref{NSC} is locally well-posed in $H^s(\T ^3)$:
For any $u_0\in H^s$ with $\mathrm{div}\,u_0=0$, there exists a unique solution $u\in C([0,T_L];H^s)\cap C((0,T_L];H^\I )$ of \eqref{NSC} on the time interval $[0,T_L]$ such that
\begin{gather}
\label{est:let} \nu T_L=\big( C(\al )\nu ^{-1}\tnorm{u_0}{H^1}\big) ^{-\frac{4\al}{4\al -3}},\\
\label{est:ls} \sup _{0<t\le T_L}\big( \tnorm{u(t)}{H^1}+(\nu t)^{\frac{1}{2}}\tnorm{u(t)}{H^{1+\al}}\big) \le C\tnorm{u_0}{H^1},\\
\label{est:1al} \nu \int _0^{T_L}\tnorm{u(t)}{H^{1+\al}}^2\,dt\le C(\al )\tnorm{u_0}{H^1}^2.
\end{gather}
Moreover, there exists $\eta =\eta (s)>0$ such that if $\tnorm{u_0}{H^1}\le \eta (1)\nu$, then the solution $u$ is global.
If $\tnorm{u_0}{H^1}\le \eta (s)\nu$, then it holds that
\eqq{\tnorm{u(t)}{H^s}\le e^{-\frac{1}{2}\nu t}\tnorm{u_0}{H^s},\qquad t\ge 0.}
\end{thm}
We give its proof in the next section.
Since we consider the subcritical problem with respect to the scaling (see Section~\ref{Fourier-Lebesgue}), time of local existence is bounded from below in terms of the size of initial data.
Note also that the result is uniform in the Coriolis parameter $\Om$.

We now state the main theorem.
\begin{thm}\label{main}
Let $\nu >0$ and $\al \in (\frac{3}{4},1]$.
For any $E>0$, there exists $\Om _0$ depending on $\al$, $\nu$ and $E$ such that for any $\Om \in \R$ with $|\Om |\ge \Om _0$ and any real-valued divergence-free mean-zero initial vector field $u_0\in H^1(\mathbb{T}^3)$ with $\tnorm{u_0}{H^1}\le E$, there exists a unique global smooth solution $u(t)\in C([0,\infty );H^1)\cap C((0,\infty );H^\I )$ of \eqref{NSC} satisfying
\eq{apri-u}{\tnorm{u(t)}{H^1}^2+\nu \int _0^t\tnorm{u(t')}{H^{1+\al}}^2\,dt'\le C(\al )E^2\Big( 1+\frac{E}{\nu}\Big) ^{C(\al )},\qquad t>0.}
Moreover, $\Om _0$ can be taken as
\eq{def:Om0}{\Om _0
=
E\exp \Big[ C(\al )\Big( \frac{E}{\nu}\Big) ^{C(\al )}\Big] .}
\end{thm}

\begin{rem}
For $a=(a_1,a_2,a_3)\in (0,\I )^3$, consider the torus 
\[ \mathbf{T}^3_a:=[0,2\pi a_1)\times [0,2\pi a_2)\times [0,2\pi a_3).\]
We say $\mathbf{T}^3_a$ \emph{regular} if $a_1=a_2=a_3$ and \emph{rational} if $a_2^2/a_1^2,a_3^2/a_1^2\in \Bo{Q}$.
Although we will mainly focus on the case of $\T ^3=[0,2\pi )^3$, our result can be extended to the case of any \emph{rational} periodic domains with a slight modification. See Remark~\ref{rem:rational} below.

We also remark that for periodic domains $\mathbf{T}^3_a$ with $a_2/a_1,a_3/a_1\in \mathbb{Q}$, global regularity under fast rotation follows immediately from the above theorem and a scaling argument.
In fact, by a suitable scaling transformation, any solution $(u,p)$ to \eqref{NSC} on $\mathbf{T}^3_a$ with $a_2/a_1,a_3/a_1\in \mathbb{Q}$ can be transformed into a solution on $\T ^3$ with rescaled initial data and rotation speed.
\end{rem}

\begin{rem}
We will obtain a global $H^1$ bound on the solution to \eqref{NSC} which is polynomial in $\nu ^{-1}$ rather than exponential. 
Even for the usual Navier-Stokes equations ($\alpha =1$) this improves the previous result of Babin \mbox{et al.} \cite{BMN2} in the case of rational periodic domains.
See Section~\ref{subsec:BMN} for related discussion.
\end{rem}

\begin{rem}\label{rem:Fourier-Lebesgue}
Let us define the function space $\mathcal F^{-1}\ell_1$ as follows:
\begin{equation*}
\mathcal{F}^{-1}\ell_1(\mathbb{T}^d):=\bigg\{ u=\sum_{n\in\mathbb{Z}^d}\hat u(n)e^{in\cdot x}\,\bigg|\, \sum_{n\in\mathbb{Z}^d}|\hat u(n)|<\infty \bigg\}. 
\end{equation*}
It is well known that $\mathcal F^{-1}\ell_1$ is continuously embedded in the space of continuous functions (in the nonperiodic case it is embedded in $BUC$, the space of bounded uniformly continuous functions).
In the spatially almost periodic case, $\mathcal F^{-1}\ell_1$ framework seems to be one of the most suitable (see \cite{GIMM2, GIMS, GIMS2, GJMY, GMN, GMY, Y} for example).
On the other hand, to control the nontrivial resonant part,  the energy method is one of the most powerful tools.
Note that, up to now, we can use the energy method only in the periodic case,
thus, controlling the nontrivial resonant part in the spatially almost periodic case is open.
See also Section \ref{Fourier-Lebesgue}.
\end{rem}

In the rest of this section, we outline the proof of the main theorem (Theorem \ref{main}). 
We basically follow the previous argument in \cite{BMN1,BMN2,BMN4,CDGG}.

The Poincar\'e propagator $\cL (\Om t)=e^{-\Om t\HP J\HP}$ is defined as the unitary group associated with the linear problem
\eqq{\p _t\Phi +\Om \HP (e_3\times \Phi )=0,\qquad \Phi \big| _{t=0}=\Phi _0\quad \text{with\quad $\mathrm{div}\,\Phi _0=0$},}
where
$\HP$ denotes the Helmholtz-Leray projection onto divergence-free vector fields; $\HP$ acts as multiplication by the matrix $\hhat{\HP}(n)$ in the Fourier space:
\eqq{\hhat{\HP}(n)=\mathrm{Id}-\big( \frac{n_in_j}{|n|^2}\big) _{1\le i,j\le 3}=\frac{1}{|n|^2}\left( \begin{matrix}
n_2^2+n_3^2 & -n_1n_2 & -n_1n_3\\
-n_2n_1 & n_1^2+n_3^2 & -n_2n_3\\
-n_3n_1 & -n_3n_2 & n_1^2+n_2^2\end{matrix}\right) .}
We see that the matrix
\eqq{\hhat{\HP}(n)J\hhat{\HP}(n)=\frac{n_3}{|n|^2}\left( \begin{matrix}0&-n_3&n_2\\ n_3&0&-n_1\\ -n_2&n_1&0\end{matrix}\right)}
has eigenvalues $\pm i\frac{n_3}{|n|},0$, and for each $n\in \Bo{Z}^3\setminus \{ 0\}$, the vectors $e^{\pm}(n)\in \Bo{C}^3$ defined by
\eqq{e^\pm (n)=\begin{cases}
\dfrac{1}{\sqrt{2}|n||n^h|}\big( n_1n_3\pm in_2|n|,\, n_2n_3\mp in_1|n|,\, -|n^h|^2\big) \quad &\text{if}~~n^h:=(n_1,n_2) \neq 0,\\
\dfrac{1}{\sqrt{2}}\big( 1,\, \mp i\,\mathrm{sgn}(n_3),\, 0\big) &\text{if}~~n^h=0\end{cases}
}
are eigenvectors corresponding to $\pm i\frac{n_3}{|n|}$ and form an orthonormal basis of 
\eqq{\Shugo{\hhat{a}\in \Bo{C}^3}{n\cdot \hhat{a}=0}=\mathrm{Ran}\,\hhat{\HP}(n).}
Therefore, the Poincar\'e propagator $\cL (\Om t)$ acts on a divergence-free and mean-free vector field $a(x)=\sum_{n\neq 0}\hhat{a}(n)e^{in\cdot x}$ as
\eqq{\big[\cL (\Om t)a\big] (x)=\sum _{n\neq 0}\sum _{\sgm \in \{ \pm \}}e^{-\sgm i\Om t\frac{n_3}{|n|}}\hhat{a}^\sgm (n)e^{in\cdot x}, \quad \hhat{a}^\sgm (n):=(\hhat{a}(n)|e^\sgm (n))e^\sgm (n),}
where $(\cdot |\cdot )$ denotes the inner product of $\Bo{C}^3$.
(Note that $a\cdot b=\sum _ja_jb_j$ whereas $(a|b)=\sum _ja_jb_j^*$, where $^*$ stands for the complex conjugate.)

\begin{rem}
It is also easy to see that
\eq{epmformula}{n\times e^\pm (n)=\pm i|n|e^\pm (n),\qquad n\in \Bo{Z}^3\setminus \{ 0\} .}
Using this, we obtain another representation of $\cL (\Om t)$ which appears in \cite[(2.10)]{BMN1}:
\eqq{\big[\cL (\Om t)a\big] (x)=\sum _{n\in \Bo{Z}^3\setminus \{ 0\}}\Big[ \cos \Big( \frac{\Om tn_3}{|n|}\Big) \hhat{a}(n)-\sin \Big( \frac{\Om tn_3}{|n|}\Big) \frac{n}{|n|}\times \hhat{a}(n)\Big] e^{in\cdot x}.}
In fact, for divergence-free and mean-zero $a$, 
\eqq{&\cos \Big( \Om t\frac{n_3}{|n|}\Big) \hhat{a}(n)-\sin \Big( \Om t\frac{n_3}{|n|}\Big) \frac{n}{|n|}\times \hhat{a}(n)\\
&=\sum _{\sgm \in \{ \pm \}}\frac{1}{2}e^{-\sgm i\Om t\frac{n_3}{|n|}}\Big( \hhat{a}(n)-\sgm i\frac{n}{|n|}\times \hhat{a}(n) \Big) \\
&=\sum _{\sgm \in \{ \pm \}}\frac{1}{2}e^{-\sgm i\Om t\frac{n_3}{|n|}}\sum _{\sgm '\in \{ \pm \}}\big( \hhat{a}(n) \big| e^{\sgm '}(n)\big) \Big( e^{\sgm '}(n)-\sgm i\frac{n}{|n|}\times e^{\sgm '}(n)\Big) \\
&=\sum _{\sgm \in \{ \pm \}}e^{-\sgm i\Om t\frac{n_3}{|n|}}\sum _{\sgm '\in \{ \pm \}}\big( \hhat{a}(n) \big| e^{\sgm '}(n)\big) \frac{e^{\sgm '}(n)+\sgm \sgm 'e^{\sgm '}(n)}{2}\\
&=\sum _{\sgm \in \{ \pm \}}e^{-\sgm i\Om t\frac{n_3}{|n|}}\hhat{a}^{\sgm}(n).}
\end{rem}

Next, we set $v(t):=\cL (-\Om t)u(t)$.
If $u(t)$ solves \eqref{NSC}, then $v$ (formally) solves
\IVP{FNS2}{&\p _tv+\nu (-\Delta )^\al v+B(\Om t;v(t),v(t))=0,\qquad t>0,\quad x\in \T ^3,\\
&v\big| _{t=0}=u_0\quad \text{with\quad $\mathrm{div}\,u_0=0$,}}
where
\eqq{B(\Om t;a,b):=\cL (-\Om t)\HP \big( \cL (\Om t)a\cdot \nabla \big) \cL (\Om t)b,}
so that
\eqq{&\big[ \F B(\Om t;a,b)\big] (n)\\
&=i\sum _{\mat{\sgm =(\sgm _1,\sgm _2,\sgm _3)\\ \in \{ \pm \} ^3}}\sum _{\mat{k,m\neq 0\\ n=k+m}}e^{-i\Om t\om ^\sgm _{nkm}}(\hhat{a}^{\sgm _1}(k)\cdot m)\big( \hhat{b}^{\sgm _2}(m)\big| e^{\sgm _3}(n)\big) e^{\sgm _3}(n),}
\[ \om ^\sgm _{nkm}:=\sgm _1\frac{k_3}{|k|}+\sgm _2\frac{m_3}{|m|}-\sgm _3\frac{n_3}{|n|}.\]
Now we decompose $B(\Om t;a,b)$ into the resonant and the non-resonant parts as

\begin{equation}\label{decomposition}
B(\Om t;a,b)=B_{R}(a,b)+B_{N\!R}(\Om t;a,b),
\end{equation}
where 
\begin{eqnarray*}
\big[ \F B_{R}(a,b)\big] (n)&:=&i\sum _{\sgm \in \{ \pm \} ^3}\sum _{\mat{n=k+m\\ \om ^\sgm _{nkm}=0}}(\hhat{a}^{\sgm _1}(k)\cdot m)\big( \hhat{b}^{\sgm _2}(m)\big| e^{\sgm _3}(n)\big) e^{\sgm _3}(n),
\end{eqnarray*}
so that
\begin{align*}
&\big[ \F B_{N\!R}(\Om t;a,b)\big] (n)\\
&=i\sum _{\sgm \in \{ \pm \} ^3}\sum _{\mat{n=k+m\\ \om ^\sgm _{nkm}\neq 0}}e^{-i\Om t\om ^\sgm _{nkm}}(\hhat{a}^{\sgm _1}(k)\cdot m)\big( \hhat{b}^{\sgm _2}(m)\big| e^{\sgm _3}(n)\big) e^{\sgm _3}(n).
\end{align*}
It is expected (and actually proved in Section~\ref{sec:error}) that only the resonant part contributes in the limit $|\Om |\to \I$.
Therefore, we need to consider the following limit equation (\emph{resonant equation}):
\IVP{limit}{
&\p _tU+\nu (-\Delta )^\al U+B_R(U(t),U(t))=0,\qquad t>0,\quad x\in \T ^3,\\
&U\big| _{t=0}=u_0\quad \text{with $\mathrm{div}\,u_0=0$.}
}

We remark that similar local existence results to Theorem \ref{prop:lwp} for the equation \eqref{FNS2}
and for the limit equation \eqref{limit} can be obtained with the identical proof.

The main task is to show existence of global
regular solutions to the resonant equation \eqref{limit}.
More precisely
we will show the following: 
\begin{prop}\label{prop:gwp-FL}
Let $\nu >0$, $\al \in (\frac{3}{4},1]$ and $u_0\in H^1(\T ^3)$ be any real-valued, divergence-free, mean-zero initial vector field.
Then, there exists a unique global solution $U\in C([0,\I );H^1)\cap L^2((0,\I );H^{1+\al})\cap C((0,\I );H^\I )$ to \eqref{limit} satisfying the same estimate as \eqref{apri-u}.
In particular, for any $s\ge 1$, the $H^s$ norm of $U(t)$ decays exponentially for large time.
\end{prop}

To prove the above proposition, we make further decomposition of $B_R$ into the 2D part and the non-trivial resonance part, as in \cite{BMN1,BMN2,BMN4,CDGG}.
For a 3D-3C (three-dimensional three-component) vector field $a=(a_1,a_2,a_3):\T ^3\to \R ^3$, we define
\begin{itemize}
\item 2D-3C vector field $\bbar{a}$ by\hx $\bbar{a}(x^h):=\dfrac{1}{2\pi}\displaystyle\int _0^{2\pi}a(x)\,dx_3$,\\
or\hx $\bbar{a}(x^h)=\displaystyle\sum\limits _{n_3=0}\hhat{a}(n)e^{in\cdot x}$,
\item 3D-3C vector field $a\osc $ by\hx $a\osc (x):=a(x)-\bbar{a}(x^h)$,\\
or\hx $a\osc (x)=\displaystyle\sum\limits _{n_3\neq 0}\hhat{a}(n)e^{in\cdot x}$,
\item 3D-2C vector field $a^h$ by\hx $a^h(x):=\big( a_1(x),\,a_2(x)\big)$.
\end{itemize}

It is easily verified that for any divergence-free and mean-zero vector fields $a,b$, 
\eqs{\bbar{B_R(\bbar{a},b\osc )}=\bbar{B_R(a\osc ,\bbar{b})}=B_R(\bbar{a},\bbar{b})\osc =0,\\
\bbar{B_R(\bbar{a},\bbar{b})}=B_R(\bbar{a},\bbar{b})=\big( \HP _h (\bbar{a}^h\cdot \nabla ^h)\bbar{b}^h,\,(\bbar{a}^h\cdot \nabla ^h)\bbar{b}_3\big) ,}
where $\HP _h$ is the 2D Helmholtz projection, and $\nabla ^h =(\p _{x_1},\p _{x_2})$. 
Note that $\mathrm{div}_h\,\bbar{u_0}^h:=\nabla ^h\cdot \bbar{u_0}^h=0$ if $\mathrm{div}\,u_0=0$.
Moreover, it is known (\cite[Theorem~3.1]{BMN1}; see also \cite[(2.24)]{BMN2}, \cite[Proposition~6.2(1)]{CDGG}, and Lemma~\ref{lem:BR} below for a proof) that 
\eqq{\bbar{B_R(a\osc ,a\osc )}=0.}
These properties imply that $\bbar{B_R(U,U)}=B_R(\bbar{U},\bbar{U})$.
Consequently, the limit equation \eqref{limit} can be decomposed into the following three equations:
\IVP{FLbarh}{
&\p _t\bbar{U}^h+\nu (-\Delta _h)^\al \bbar{U}^h+\HP _h(\bbar{U}^h\cdot \nabla ^h)\bbar{U}^h=0,\qquad t>0,\quad x\in \T ^2,\\
&\bbar{U}^h\big| _{t=0}=\bbar{u_0}^h\quad \text{with $\mathrm{div}_h\,\bbar{u_0}^h=0$},}
\IVP{FLbar3}{&\p _t\bbar{U}_3+\nu (-\Delta _h)^\al \bbar{U}_3+(\bbar{U}^h\cdot \nabla ^h)\bbar{U}_3=0,\qquad t>0,\quad x\in \T ^2,\\
&\bbar{U}_3\big| _{t=0}=\bbar{u_0}_{,3},}
\IVP{FLosc}{&\p _tU\osc +\nu (-\Delta )^\al U\osc \\
&\hx +B_R(\bbar{U},U\osc )+B_R(U\osc ,\bbar{U})+B_R(U\osc ,U\osc )=0,\quad t>0,~~ x\in \T ^3,\\
&U\osc \big| _{t=0}=u_{0,\mathrm{osc}}\quad \text{with $\mathrm{div}\,u_{0,\mathrm{osc}}=0$,}}
where
$(-\Delta_h)^\alpha f:=\mathcal{F}^{-1}[(n_1^2+n_2^2)^\alpha\hat f]$.

The $H^1$ energy estimate for the 2D part $\bbar{U}(t)$ can be obtained straightforwardly
: 
\begin{equation}\label{energy2}
\|\bbar U(t)\|_{H^1}^2+\nu \int_0^t\|\bbar U(t')\|_{H^{1+\alpha}}^2dt'\leq C(\|\bbar U(0)\|_{H^1})<\infty .
\end{equation}
The key is to 
control the following norm globally in time:
\begin{equation}\label{energy1}
\| U\osc (t)\|_{H^1}^2+\nu \int_0^t\| U\osc (t')\|_{H^{1+\alpha}}^2dt'.
\end{equation}
In order to control the above quantity in the weak viscosity case ($\al <1$), we essentially use a new estimate on non-trivial resonant three-wave interactions (Lemma~\ref{lem:combi} below).
We prove this estimate in Section \ref{non-trivial resonant part} by using some tools from elementary number theory, which is the crucial idea in this paper.
However, this kind of argument is only available for the case of \emph{regular} (or \emph{rational}) torus.
As a result, our main theorem is also restricted to that case. 
Once we get the above energy estimates, we will be able to deduce Proposition~\ref{prop:gwp-FL}.
See Section \ref{GWP for the limit equation} for details.

To prove the main theorem, 
%
%
it suffices to ensure under the large Coriolis parameter assumption that solutions to the original equation and the resonant equation that coincide at $t=0$ stay close to each other until an arbitrarily given time $t=T$, which also means that we can obtain the existence theorem in $[0,T]$.
To this end, it suffices to control the non-resonant part in \eqref{decomposition}.
More precisely, our task is to estimate the difference $w(t):=v(t)-U(t)$, which satisfies
\begin{equation*}
\begin{split}
\left\{ 
\begin{split}
&\p _tw+\nu (-\Delta )^\al w+B_R(w,v)+B_R(U,w)+B_{N\!R}(\Om t; v,v)=0,\quad t>0,\\
&w\big| _{t=0}=0.
\end{split} \right.
\end{split} 
\end{equation*}
Let $\ti{E}$ be the global upper bound of $U(t)$ in $H^1$ given in Proposition~\ref{prop:gwp-FL}, and let $\ti{T}_L$ be a local existence time of the $H^1$ solution to \eqref{FNS2} of size $2\ti{E}$.
The following lemma enables us to control the non-resonant part $B_{N\!R}$:
\begin{lem}\label{lem:est-Om'}
For any $\de >0$, there exists $\Om _0=\Om _0(\de , \al ,\nu ,E)>0$ such that the following holds for $|\Om |\ge \Om _0$.
Let $T\ge 0$, and assume that the solution $v$ of \eqref{FNS2} exists on $[0,T]$ with $\tnorm{v}{L^\I ([0,T];H^1)}^2+\nu \tnorm{v}{L^2([0,T];H^{1+\al})}^2\le (2\ti{E})^2$.
Then, we have
\eqq{&\Big| \int _0^t\LR{B_{N\!R}(\Om t'; v(t'),v(t'))\,,\,w(t')}_{H^1}\,dt'\Big| \\
&\le \de +\frac{1}{4}\Big( \tnorm{w(t)}{H^1}^2+\nu \int _0^t\tnorm{w(t')}{H^{1+\al}}^2\,dt'\Big) ,\qquad t\in [0,T+\ti{T}_L].}
\end{lem}

Roughly saying, we can control the contribution from the non-resonant forcing term by an arbitrarily small constant $\delta$.
By the above lemma, which will be restated as Lemma~\ref{lem:est-Om} and proved in Section \ref{sec:error}, we can prove the main theorem (Theorem\ref{main}).
For the precise argument, see Section \ref{sec:error}.

\vspace{0.5cm}
\noindent
{\bf Acknowledgments.}\ 
We would like to thank the referees for carefully reading the previous version of this article and giving us useful suggestions and comments.
The first author was partially supported by 
Grant-in-Aid for Young Scientists (B), No.~24740086 and No.~16K17626,
Japan Society for the Promotion of Science.
The second author was also partially supported by 
Grant-in-Aid for Young Scientists (A), No.~17H04825 and Grant-in-Aid for Scientific Research (B), No.~17H02860,
Japan Society for the Promotion of Science.


\bigskip
\section{Proof of local well-posedness}

In this section, we shall establish local well-posedness for \eqref{NSC}, i.e., Theorem~\ref{prop:lwp}.
Let us consider the corresponding integral equation:
\eq{I-FNS}{u(t)=S(t)u_0-\int _0^tS(t-t')\HP (u(t')\cdot \nabla )u(t')\,dt',\qquad S(t):=e^{-\nu t(-\Delta )^\al}\cL (\Om t).}
Note that $\{ S(t)\} _{t\ge 0}$ is a continuous semigroup of contractions on $H^s$ for any $s\in \R$.

\begin{lem}\label{lem:smoothing}
Let $\nu ,\al >0$, $\Om \in \R$, $s\in \R$ and $\theta \ge 0$.
Then, for divergence-free and mean-zero vector fields $a$, it holds that
\eqq{\norm{S(t)a}{H^{s+\th}}\lec (\nu t)^{-\frac{\th}{2\al}}\tnorm{a}{H^s},\qquad t>0.}
The implicit constant depends only on $\frac{\th}{2\al}$ and is independent of $\nu ,t, s, \Om$. 
\end{lem}
\begin{rem}
In the following argument, we apply this lemma only with $0\le \th <2\al$.
Then, we can take the implicit constant independent of $\al$ and $\th$.
\end{rem}
\begin{proof}
We see that
\eqq{\norm{S(t)a}{H^{s+\th}}&=\norm{|n|^{\th}e^{-\nu t|n|^{2\al}}\cdot |n|^s\big[ \F \cL (\Om t)a\big] (n)}{\ell ^2(\Bo{Z}^3)}\\
&\lec \norm{|n|^{\th}e^{-\nu t|n|^{2\al}}}{\ell ^\I (\Bo{Z}^3)}\tnorm{a}{H^s}\le \big[ \sup _{r>0}r^{\frac{\th}{2\al}}e^{-r}\big] (\nu t)^{-\frac{\th}{2\al}}\tnorm{a}{H^s}.\qedhere
}
\end{proof}

\begin{proof}[Proof of Theorem~\ref{prop:lwp}]
We will apply fixed point argument in the Banach space
\eqs{X^s_T:=\Shugo{u\in C([0,T];H^s)\cap C((0,T];H^{s+\al})}{\tnorm{u}{X_{T}^s}<\I },\\
\tnorm{u}{X_{T}^s}:=\sup _{0<t\le T}\Big( \tnorm{u(t)}{H^s}+(\nu t)^{\frac{1}{2}}\tnorm{u(t)}{H^{s+\al}}\Big) .}

By Lemma~\ref{lem:smoothing}, we have
\eqq{\norm{S(t)u_0}{X_T^s}\lec \tnorm{u_0}{H^s},\qquad T>0.}
Moreover, by continuity of $S(t)$ we see that $S(t)u_0\in C([0,\I );H^s)\cap C((0,\I );H^\I )$.

For the Duhamel term, let $u,v\in X^s_T$.
We first consider the estimate when $s=1$.
By using Lemma~\ref{lem:smoothing} and the Sobolev inequality (see Lemma~\ref{lem:Sobolev} below), we have
\eq{est:nonl-1}{&\norm{\int _0^tS(t-t')\HP (u(t')\cdot \nabla )v(t')\,dt'}{H^1}\\
&\lec \int _0^t\big[ \nu (t-t')\big] ^{-\frac{3}{4\al}}\norm{(u(t')\cdot \nabla )v(t')}{H^{-\frac{1}{2}}}\,dt'\\
&\lec \int _0^t\big[ \nu (t-t')\big] ^{-\frac{3}{4\al}}\tnorm{u(t')}{H^{1}}\tnorm{\nabla v(t')}{L^2}\,dt'\\
&\lec _\al \nu ^{-1}(\nu T)^{1-\frac{3}{4\al}}\tnorm{u}{L^\I ([0,T];H^1)}\tnorm{v}{L^\I ([0,T];H^1)}}
for $0<t\le T$, where we have used the assumption $\al >\frac{3}{4}$.
Similarly,
\eqq{&(\nu t)^{\frac{1}{2}}\norm{\int _0^tS(t-t')\HP (u(t')\cdot \nabla )v(t')\,dt'}{H^{1+\al}}\\
&\lec (\nu t)^{\frac{1}{2}}\int _0^t\big[ \nu (t-t')\big] ^{-\frac{3}{4\al}}\norm{(u(t')\cdot \nabla )v(t')}{H^{\al -\frac{1}{2}}}\,dt'\\
&\lec (\nu t)^{\frac{1}{2}}\int _0^t\big[ \nu (t-t')\big] ^{-\frac{3}{4\al}}\tnorm{u(t')}{H^1}\tnorm{\nabla v(t')}{H^{\al}}\,dt'\\
&\lec _\al \nu ^{-1}(\nu T)^{1-\frac{3}{4\al}}\tnorm{u}{L^\I ([0,T];H^1)}\tnorm{(\nu \cdot )^{\frac{1}{2}}v(\cdot )}{L^\I ((0,T];H^{1+\al})}}
for $0<t\le T$, and therefore,
\eqq{\norm{\int _0^tS(t-t')\HP (u(t')\cdot \nabla )v(t')\,dt'}{X^1_T}\le C(\al )\nu ^{-1}(\nu T)^{1-\frac{3}{4\al}}\tnorm{u}{X_T^1}\tnorm{v}{X_T^1}.}
The estimate for general $s\ge 1$ can be deduced from that for $s=1$ as
\eqq{&\norm{\int _0^tS(t-t')\HP (u(t')\cdot \nabla )v(t')\,dt'}{X^s_T}\\
&\le C(s, \al )\nu ^{-1}(\nu T)^{1-\frac{3}{4\al}}\big( \tnorm{u}{X_T^s}\tnorm{v}{X_T^1}+\tnorm{u}{X_T^1}\tnorm{v}{X_T^s}\big) .}
(Note that the $s$-dependence may come into the estimate when we divide the derivative $|\nabla|^{s-1}$, since $|n|^{s-1}\le 2^{\max \{ s-2,0\}}(|k|^{s-1}+|n-k|^{s-1})$.)

We next show that for $u,v\in X^s_T$
\[ I(u,v)(t):=\int _0^tS(t-t')\HP (u(t')\cdot \nabla )v(t')\,dt'\in C([0,T];H^s)\cap C((0,T];H^{s+\al}),\]
namely, continuity in $t$.
Set $F(t):=\HP (u(t)\cdot \nabla )v(t)$. 
As seen in the above estimates, we have $F\in C([0,T];H^{s-\frac{3}{2}})\cap C((0,T];H^{s+\al -\frac{3}{2}})$.
For $t\in [0,T)$ and $0<\de \ll 1$, we have
\eqq{&\tnorm{I(u,v)(t+\de )-I(u,v)(t)}{H^s}\\
&\le \int _t^{t+\de}\norm{S(t+\de -t')F(t')}{H^s}\,dt'+\int _0^t\norm{S(t-t')\big[ S(\de )F(t')-F(t')\big]}{H^s}\,dt'\\
&\lec _\nu \de ^{1-\frac{3}{4\al}}\tnorm{F}{L^\I ([0,T];H^{s-\frac{3}{2}})}+\int _0^t (t-t')^{-\frac{3}{4\al}}\tnorm{S(\de )F(t')-F(t')}{H^{s-\frac{3}{2}}}\,dt'.
}
The right-hand side tends to $0$ as $\de \to 0$ by the dominated convergence theorem.
Similarly, for $t\in (0,T]$, 
\eqq{&\tnorm{I(u,v)(t-\de )-I(u,v)(t)}{H^s}\\
&\le \int _0^\de \norm{S(t-t')F(t')}{H^s}\,dt'+\int _{\de}^t\norm{S(t-t')\big[ F(t')-F(t'-\de )]}{H^s}\,dt'\\
&\lec _\nu \de (t-\de )^{-\frac{3}{4\al}}\tnorm{F}{L^\I ([0,T];H^{s-\frac{3}{2}})}+\int _\de ^t (t-t')^{-\frac{3}{4\al}}\tnorm{F(t')-F(t'-\de )}{H^{s-\frac{3}{2}}}\,dt'\to 0
}
as $\de \to 0$.
Next, for any $0<T'\!\ll 1$ the above argument and $F\in C([T',T];H^{s+\al -\frac{3}{2}})$ show that $\int _{T'}^tS(t-t')F(t')\,dt'\in C([T',T];H^{s+\al})$.
On the other hand, for $t\in [3T',T]$ and $\de \in [-T',T']$, 
\eqq{&\norm{\int _0^{T'}S(t+\de -t')F(t')\,dt'-\int _0^{T'}S(t-t')F(t')\,dt'}{H^{s+\al}}\\
&\le \int _0^{T'}\norm{\big[ S(t+\de -T'-t')-S(t-T'-t')\big] S(T')F(t')}{H^{s+\al}}\,dt'.
}
By the dominated convergence theorem and $S(T')F(\cdot )\in L^\I ([0,T'];H^{s+\al})$, the right-hand side tends to $0$ as $\de \to 0$.
Since $0<T'\ll 1$ is arbitrary, we conclude $I(u,v)\in C([0,T];H^s)\cap C((0,T];H^{s+\al})$.

Then, the contraction mapping principle can be applied if we take $T=T_s>0$ so that
\eq{def:T_L}{\nu ^{-1}(\nu T_s)^{1-\frac{3}{4\al}}\tnorm{u_0}{H^s}\ll _{s,\al}1.}
In particular, we obtain existence (and continuous dependence on initial data) of a solution on $[0,T_s]$ belonging to $X_{T_s}^s\subset C([0,T_s];H^s)\cap C((0,T_s];H^{s+\al})$.
The estimate \eqref{est:nonl-1} can be used to show uniqueness of solutions in $C([0,T];H^1)$.
Smoothness of the solution is verified by iterating the above construction of local solutions which gains $\al$ regularity, together with the uniqueness.

To show that the local existence time depends only on the $H^1$ norm of $u_0\in H^s$, it suffices to apply the above construction also with $s=1$, so that by uniqueness we have $u\in X^s_{T_s}\cap X_{T_1}^1\subset C([0,T_1];H^s)\cap C((0,T_1];H^\I )$.
Defining $T_L:=T_1$, we have \eqref{est:let} and \eqref{est:ls}.

From the Sobolev inequality (Lemma~\ref{lem:Sobolev}), we have
\eqq{\norm{\HP (f\cdot \nabla )g}{L^2}\le \tnorm{f}{H^1}\tnorm{g}{H^{\frac{3}{2}}}.}
This and interpolation show that, for $\al \ge \frac{3}{4}$ and $s\ge \frac{3}{4}$,
\eqq{&|\LR{\HP (f\cdot \nabla )f\,,\,f}_{H^s}|\le \norm{\HP (f\cdot \nabla )f}{H^{s-\frac{3}{4}}}\tnorm{f}{H^{s+\frac{3}{4}}}\\
&\lec _s \big( \tnorm{f}{H^{s+\frac{1}{4}}}\tnorm{f}{H^{\frac{3}{2}}}+\tnorm{f}{H^{1}}\tnorm{f}{H^{s+\frac{3}{4}}}\big) \tnorm{f}{H^{s+\frac{3}{4}}}\lec  \tnorm{f}{H^{1}}\tnorm{f}{H^{s+\frac{3}{4}}}^2 .}
By skew-symmetry of the Coriolis term in \eqref{NSC}, the standard $H^s$-energy estimate with the above inequality implies that
\eq{ineq:Hsenergy}{\frac{d}{dt}\tnorm{u(t)}{H^s}^2+2\nu \tnorm{u(t)}{H^{s+\al}}^2&\le C_*(s)\tnorm{u(t)}{H^1}\tnorm{u(t)}{H^{s+\frac{3}{4}}}^2\\
&\le C_*(s)\tnorm{u(t)}{H^1}\tnorm{u(t)}{H^{s+\al}}^2}
for any smooth solutions $u(t)$ of \eqref{NSC}, where the constant $C_*(s)>0$ depends only on $s$.

From \eqref{est:ls} and interpolation, the solution $u\in X^1_{T_L}$ obtained above satisfies
\eqq{\tnorm{u(t)}{H^1}\tnorm{u(t)}{H^{\frac{7}{4}}}^2\le C(\nu t)^{-\frac{3}{4\al}}\tnorm{u_0}{H^1}^3.\,\qquad t\in (0,T_L].}
Hence, by integrating \eqref{ineq:Hsenergy} with $s=1$, we have
\eqq{\tnorm{u(t)}{H^1}^2+2\nu \int _0^t\tnorm{u(t')}{H^{s+\al}}^2\,dt'\le \tnorm{u_0}{H^1}^2+C(\al )\nu ^{-1}(\nu t)^{1-\frac{3}{4\al}}\tnorm{u_0}{H^1}^3}
for $t\in (0,T_L]$ if $\al >\frac{3}{4}$.
Together with \eqref{def:T_L}, we obtain \eqref{est:1al}.

Finally, we prove the global existence for small initial data.
Assume that the initial data $u_0$ satisfies the smallness condition
\eqq{\tnorm{u_0}{H^1}\le \frac{\nu}{2C_*(1)},}
and define 
\eqq{T_*:=\sup \Shugo{T\ge 0}{\text{the solution $u(t)$ exists and $\tnorm{u(t)}{H^1}\le \frac{\nu}{C_*(1)}$ on $[0,T]$}}.}
By the local theory established above, we have $T_*>0$.
Furthermore, \eqref{ineq:Hsenergy} with $s=1$ shows that
\eq{ineq:smalldata}{\frac{d}{dt}\tnorm{u(t)}{H^s}^2+\nu \tnorm{u(t)}{H^{s+\al}}^2\le 0,\qquad t\in (0,T_*)}
with $s=1$.
In particular, $\tnorm{u(t)}{H^1}$ is decreasing and thus $T_*=\I$.
If we further assume that
\eqq{\Big( \tnorm{u(t)}{H^1}\le \Big)~\tnorm{u_0}{H^1}\le \frac{\nu}{2C_*(s)},}
then \eqref{ineq:smalldata} holds for this choice of $s$, which combined with $\tnorm{u}{H^s}\le \tnorm{u}{H^{s+\al}}$ implies the exponential decay of $\tnorm{u(t)}{H^s}$ as $t\to \I$.
\end{proof}


\bigskip
\section{Properties of the nonlinear term in the resonant equation}\label{kl}

In this section we recall some cancellation properties of the nonlinear term $B_R$ in the resonant equation. 
These properties have been essentially proved in the previous works \cite{BMN1,BMN2,BMN4,CDGG}, but we shall present a proof of them for the sake of completeness.

\begin{lem}[\mbox{cf.} Theorem~3.1 in \cite{BMN1}; see also Proposition~6.2~(1) in \cite{CDGG}]\label{lem:BR}
For any 

\noindent
divergence-free mean-zero smooth vector field $a$, we have $\bbar{B_R(a\osc ,a\osc )}=0$.
\end{lem}

\begin{proof}
We first notice that, under $n_3=0$ and $k_3\neq 0$, 
\eqq{\om ^\sgm _{nk(n-k)}=\sgm _1\frac{k_3}{|k|}+\sgm _2\frac{-k_3}{|n-k|}=0\qquad \Longleftrightarrow \qquad \sgm _1=\sgm _2,\quad |k|=|n-k|.}
Hence,
\eqq{&\bbar{B_R(a\osc ,a\osc )}(x)\\
&=i\sum _{\mat{n\in \Bo{Z}^3\setminus \{ 0\} \\ n_3=0}}e^{in\cdot x}\sum _{(\sgm ,\sgm _3)\in \{ \pm\} ^2}\sum _{\mat{k\in \Bo{Z}^3\setminus \{ 0\} \\ k_3\neq 0\\ |k|=|n-k|}}\big( \hhat{a}(k)\big| e^\sgm (k)\big) \,\big( \hhat{a}(n-k)\big| e^\sgm (n-k)\big) \\
&\hspace{50pt}\cdot \Big( \big[ e^\sgm (k)\cdot (n-k)\big] e^\sgm (n-k)\Big| e^{\sgm _3}(n)\Big) e^{\sgm _3}(n)\\
&=\frac{i}{2}\sum _{\mat{n\in \Bo{Z}^3\setminus \{ 0\} \\ n_3=0}}e^{in\cdot x}\sum _{(\sgm ,\sgm _3)\in \{ \pm\} ^2}\sum _{\mat{k\in \Bo{Z}^3\setminus \{ 0\} \\ k_3\neq 0\\ |k|=|n-k|}}\big( \hhat{a}(k)\big| e^\sgm (k)\big) \,\big( \hhat{a}(n-k)\big| e^\sgm (n-k)\big) \\
&\hspace{30pt}\cdot \Big( \big[ e^\sgm (k)\cdot (n-k)\big] e^\sgm (n-k)+\big[ e^\sgm (n-k)\cdot k\big] e^\sgm (k)\Big| e^{\sgm _3}(n)\Big) e^{\sgm _3}(n).}
Hence, all we have to do is to show
\eq{claim:lem1}{\Big( \big[ e^\sgm (k)\cdot (n-k)\big] e^\sgm (n-k)+\big[ e^\sgm (n-k)\cdot k\big] e^\sgm (k)\Big| e^{\sgm _3}(n)\Big) =0}
for any $\sgm ,\sgm _3\in \{ \pm \}$, $n,k\in \Bo{Z}^3\setminus \{ 0\}$ such that $n_3=0$, $k_3\neq 0$, $|k|=|n-k|$.

By the formula in vector analysis, we have
\eqq{&e^\sgm (k)\times \big[ (n-k)\times e^\sgm (n-k)\big] \\
&\hx =\big[ e^\sgm (k)\cdot e^\sgm (n-k)\big] (n-k)-\big[ e^\sgm (k)\cdot (n-k)\big] e^\sgm (n-k),\\
&e^\sgm (n-k)\times \big[ k\times e^\sgm (k)\big] \\
&\hx =\big[ e^\sgm (n-k)\cdot e^\sgm (k)\big] k-\big[ e^\sgm (n-k)\cdot k\big] e^\sgm (k).}
By \eqref{epmformula} and the assumption, we have
\eqq{&e^\sgm (k)\times \big[ (n-k)\times e^\sgm (n-k)\big] =\sgm i|n-k|\big[ e^\sgm (k)\times e^\sgm (n-k)\big] \\
&=-\sgm i|k|\big[ e^\sgm (n-k)\times e^\sgm (k)\big] =-e^\sgm (n-k)\times \big[ k\times e^\sgm (k)\big] .}
Therefore, we have
\eqq{\big[ e^\sgm (k)\cdot (n-k)\big] e^\sgm (n-k)+\big[ e^\sgm (n-k)\cdot k\big] e^\sgm (k)=\big[ e^\sgm (k)\cdot e^\sgm (n-k)\big] n.}
Since $\hhat{\HP}(n)n=0$, the left-hand side belongs to $(\mathrm{Ran}\,\hhat{\HP}(n))^\perp$, and \eqref{claim:lem1} holds.
\end{proof}

\begin{lem}[\mbox{cf.} Theorem~5.3 in \cite{BMN1}; see also Proposition~6.2~(2) in \cite{CDGG}]\label{lem:BR2}
Let $s\ge 0$ and $a,b$ be any divergence-free and mean-zero smooth vector fields.
Assume that $b$ is real-valued.
Then, we have
\eqq{\LR{B_R(\bbar{a},b\osc )\,,\,b\osc}_{H^s}=\LR{B_R(b\osc ,\bbar{b})\,,\,b\osc}_{H^s}=\LR{B_R(a\osc ,b\osc )\,,\,b\osc}_{L^2}=0.}
\end{lem}
\begin{proof}
Let us first consider $B_R(\bbar{a},b\osc )$.
Similarly to the proof of Lemma~\ref{lem:BR}, we see that
\eqq{&\LR{B_R(\bbar{a},b\osc )\,,\,b\osc}_{H^s}\\
&=i\sum _{(\sgm _1,\sgm )\in \{ \pm \} ^2}\sum _{\mat{n,k\in \Bo{Z}^3\setminus \{ 0\} \\ n_3\neq 0=k_3\\ |n|=|n-k|}}\big[ \hhat{a}^{\sgm _1}(k)\cdot (n-k)\big] \big[ \hhat{b}^\sgm (n-k)\cdot e^\sgm (n)^*\big] \big[ e^\sgm (n)\cdot |n|^{2s}\hhat{b}(n)^*\big] \\
&=i\sum _{(\sgm _1,\sgm )\in \{ \pm \} ^2}\sum _{\mat{n,k\in \Bo{Z}^3\setminus \{ 0\} \\ n_3\neq 0=k_3\\ |n|=|n-k|}}|n|^{2s}\big[ \hhat{a}^{\sgm _1}(k)\cdot (n-k)\big] \big[ \hhat{b}^\sgm (n-k)\cdot \hhat{b}^\sgm (n)^*\big] .}

Since $b$ is real-valued, $\hhat{b}(n)^*=\hhat{b}(-n)$ and thus $\hhat{b}^\sgm (n)^*=\hhat{b}^\sgm (-n)$ for any $n$.
By a change of variables $n\mapsto n':=k-n$,
\eqq{&\LR{B_R(\bbar{a},b\osc )\,,\,b\osc}_{H^s}\\
&=i\sum _{(\sgm _1,\sgm )\in \{ \pm \} ^2}\sum _{\mat{n,k\in \Bo{Z}^3\setminus \{ 0\} \\ n_3\neq 0=k_3\\ |n|=|n-k|}}|n|^{2s}\big[ \hhat{a}^{\sgm _1}(k)\cdot (n-k)\big] \big[ \hhat{b}^\sgm (-n)\cdot \hhat{b}^\sgm (k-n)^*\big] \\
&=i\sum _{(\sgm _1,\sgm )\in \{ \pm \} ^2}\sum _{\mat{n',k\in \Bo{Z}^3\setminus \{ 0\} \\ n_3'\neq 0=k_3\\ |k-n'|=|n'|}}|k-n'|^{2s}\big[ \hhat{a}^{\sgm _1}(k)\cdot (-n')\big] \big[ \hhat{b}^\sgm (n'-k)\cdot \hhat{b}^\sgm (n')^*\big] \\
&=-i\sum _{(\sgm _1,\sgm )\in \{ \pm \} ^2}\sum _{\mat{n',k\in \Bo{Z}^3\setminus \{ 0\} \\ n_3'\neq 0=k_3\\ |n'|=|n'-k|}}|n'|^{2s}\big[ \hhat{a}^{\sgm _1}(k)\cdot n'\big] \big[ \hhat{b}^\sgm (n'-k)\cdot \hhat{b}^\sgm (n')^*\big] .}
Since $e^{\sgm _1}(k)\cdot k=0$, we have
\eqq{&2\LR{B_R(\bbar{a},b\osc )\,,\,b\osc}_{H^s}\\
&=-i\sum _{(\sgm _1,\sgm )\in \{ \pm \} ^2}\sum _{\mat{n,k\in \Bo{Z}^3\setminus \{ 0\} \\ n_3\neq 0=k_3\\ |n|=|n-k|}}|n|^{2s}\big[ \hhat{a}^{\sgm _1}(k)\cdot k\big] \big[ \hhat{b}^\sgm (n-k)\cdot \hhat{b}^\sgm (n)^*\big] =0.}

Next, we consider $B_R(b\osc ,\bbar{b})$.
In a similar manner, by a change of variables $n\mapsto n':=k-n$ and $e^{\pm}(k)\cdot k=0$,
\eqq{&\LR{B_R(b\osc ,\bbar{b})\,,\,b\osc}_{H^s}\\
&=i\sum _{(\sgm ,\sgm _2)\in \{ \pm \} ^2}\sum _{\mat{n,k\in \Bo{Z}^3\setminus \{ 0\} \\ n_3=k_3\neq 0\\ |n|=|k|}}|n|^{2s}\big[ \hhat{b}^{\sgm}(k)\cdot (n-k)\big] \big[ \hhat{b}^{\sgm _2}(n-k)\cdot \hhat{b}^\sgm (n)^*\big] \\
&=i\sum _{(\sgm ,\sgm _2)\in \{ \pm \} ^2}\sum _{\mat{n,k\in \Bo{Z}^3\setminus \{ 0\} \\ n_3=k_3\neq 0\\ |n|=|k|}}|k|^s|n|^{s}\big[ \hhat{b}^{\sgm}(k)\cdot (n-k)\big] \big[ \hhat{b}^\sgm (-n)\cdot \hhat{b}^{\sgm _2}(k-n)^*\big] \\
&=i\sum _{(\sgm ,\sgm _2)\in \{ \pm \} ^2}\sum _{\mat{n',k\in \Bo{Z}^3\setminus \{ 0\} \\ n_3'=0\neq k_3\\ |k-n'|=|k|}}|k|^s|n'-k|^{s}\big[ \hhat{b}^{\sgm}(k)\cdot (-n')\big] \big[ \hhat{b}^\sgm (n'-k)\cdot \hhat{b}^{\sgm _2}(n')^*\big] \\
&=-i\sum _{(\sgm ,\sgm _2)\in \{ \pm \} ^2}\sum _{\mat{n',k\in \Bo{Z}^3\setminus \{ 0\} \\ n_3'=0\neq k_3\\ |k-n'|=|k|}}\big[ |k|^s\hhat{b}^{\sgm}(k)\cdot (n'-k)\big] \big[ |n'-k|^{s}\hhat{b}^\sgm (n'-k)\cdot \hhat{b}^{\sgm _2}(n')^*\big] \\
&=-\LR{\bbar{B_R(|\nabla |^sb\osc ,|\nabla |^sb\osc )}\,,\,b}_{L^2}=0,}
where we have applied Lemma~\ref{lem:BR} at the last equality.

Finally, since $b$ is real-valued,
\eqq{&\LR{B_R(a\osc ,b\osc )\,,\,b\osc}_{L^2}\\
&=i\sum _{(\sgm _1,\sgm _2,\sgm _3)\in \{ \pm \} ^3}\sum _{\mat{n,k\in \Bo{Z}^3\setminus \{ 0\} \\ n_3k_3(n_3-k_3)\neq 0}}\big[ \hhat{a}^{\sgm _1}(k)\cdot (n-k)\big] \big[ \hhat{b}^{\sgm _2}(n-k)\cdot \hhat{b}^{\sgm _3}(n)^*\big] \\
&=i\sum _{(\sgm _1,\sgm _2,\sgm _3)\in \{ \pm \} ^3}\sum _{\mat{n,k\in \Bo{Z}^3\setminus \{ 0\} \\ n_3k_3(n_3-k_3)\neq 0}}\big[ \hhat{a}^{\sgm _1}(k)\cdot (n-k)\big] \big[ \hhat{b}^{\sgm _3}(-n)\cdot \hhat{b}^{\sgm _2}(k-n)^*\big] \\
&=-i\sum _{(\sgm _1,\sgm _2',\sgm _3')\in \{ \pm \} ^3}\sum _{\mat{n',k\in \Bo{Z}^3\setminus \{ 0\} \\ n_3'k_3(n_3'-k_3)\neq 0}}\big[ \hhat{a}^{\sgm _1}(k)\cdot n'\big] \big[ \hhat{b}^{\sgm _2'}(n'-k)\cdot \hhat{b}^{\sgm _3'}(n')^*\big] ,}
where we have changed the variables as $(\sgm _2,\sgm _3,n)\mapsto (\sgm _2',\sgm _3',n'):=(\sgm _3,\sgm _2,k-n)$.
Therefore, 
\eqq{&2\LR{B_R(a\osc ,b\osc )\,,\,b\osc}_{L^2}\\
&=-i\sum _{(\sgm _1,\sgm _2,\sgm _3)\in \{ \pm \} ^3}\sum _{\mat{n,k\in \Bo{Z}^3\setminus \{ 0\} \\ n_3k_3(n_3-k_3)\neq 0}}\big[ \hhat{a}^{\sgm _1}(k)\cdot k\big] \big[ \hhat{b}^{\sgm _2}(n-k)\cdot \hhat{b}^{\sgm _3}(n)^*\big] =0,}
as desired.
\end{proof}

\bigskip
\section{A priori estimate and global existence for the limit equation}\label{GWP for the limit equation}

In this section, we shall prove Proposition~\ref{prop:gwp-FL}.
By the local theory, we can solve the limit equation \eqref{limit} in $H^1$ for a short time and the solution immediately becomes smooth.
Therefore, for the global existence, it suffices to derive global a~priori $H^1$ estimate on smooth solutions of \eqref{limit}$=$\eqref{FLbarh}$+$\eqref{FLbar3}$+$\eqref{FLosc}.

Let $\al \in (\frac{3}{4},1]$, $u_0\in H^1(\T ^3)$ be a divergence-free, mean-zero initial vector field and $M:=\tnorm{u_0}{L^2}$, $E:=\tnorm{u_0}{H^1}$.

\subsection{2D horizontal part \eqref{FLbarh}}
Note that this is the usual 2D incompressible Navier-Stokes equations.
However, when $\al$ is strictly less than one, we need to consider equation for the vorticity $\omega =\nabla _h^\perp \bbar{U}^h:=\p _{x_1}\bbar{U}_2-\p _{x_2}\bbar{U}_1$:
\IVP{Fvor}{&\p _t\omega +\nu (-\Delta _h)^\al \omega +(\bbar{U}^h\cdot \nabla ^h)\omega =0,\qquad t>0,\quad x\in \T ^2,\\
&\omega \big| _{t=0}=\nabla _h^\perp \bbar{u_0}^h.}
Note that $\bbar{U}^h$ can be recovered from $\omega$ by the Biot-Savart law $\bbar{U}^h=-(-\Delta _h)^{-1}\nabla ^\perp _h\omega$ and 
\eqq{\tnorm{\omega}{H^{s}}\sim \tnorm{\bbar{U}^h}{H^{s+1}},\qquad s\in \R ,}
whenever $\omega$ is mean-zero.
The standard $L^2$-energy estimate for \eqref{Fvor} yields that
\eqq{\frac{d}{dt}\tnorm{\omega (t)}{L^2}^2+2\nu \tnorm{\omega (t)}{H^\al }^2\le 0,}
or
\eq{apri-Ubarh}{\tnorm{\bbar{U}^h(t)}{H^1}^2+\nu \int _0^t\tnorm{\bbar{U}^h(t')}{H^{1+\al}}^2\,dt'\le C\tnorm{\bbar{u_0}^h}{H^1}^2\le CE^2,\qquad t>0.}

\subsection{2D vertical part \eqref{FLbar3}}\label{subsec:U3}
We begin with the easy $L^2$-energy estimate for \eqref{FLbar3}: 
\eq{apri-Ubar3-0}{\norm{\bbar{U}_3(t)}{L^2}^2+2\nu \int _0^t\norm{\bbar{U}_3(t')}{H^{\al}}^2\,dt'\le \norm{\bbar{u_0}_{,3}}{L^2}^2\le M^2,\qquad t>0.}
For the $H^1$-energy estimate, we see that the 2D Sobolev inequality and interpolation argument yields that
\eqq{\big| \LR{(\bbar{U}^h\cdot \nabla _h)\bbar{U}_3\,,\,\bbar{U}_3}_{H^1}\big| &=\big| \LR{\nabla _h\bbar{U}^h\,,\,\nabla _h\bbar{U}_3\otimes \nabla _h\bbar{U}_3}_{L^2}\big| \\
&\le \tnorm{\nabla _h\bbar{U}^h}{L^2}\tnorm{\nabla _h\bbar{U}_3}{L^4}^2\lec \tnorm{\bbar{U}^{h}}{H^{1}}\tnorm{\bbar{U}_3}{H^{\frac{3}{2}}}^2\\
&\le \tnorm{\bbar{U}^{h}}{H^{1}}\tnorm{\bbar{U}_3}{H^{1+\al}}^{3-2\al}\tnorm{\bbar{U}_3}{H^\al}^{2\al -1}\\
&\le \nu \tnorm{\bbar{U}_3}{H^{1+\al}}^2+C\nu ^{-\frac{3-2\al}{2\al -1}}\tnorm{\bbar{U}^h}{H^{1}}^{\frac{2}{2\al -1}}\tnorm{\bbar{U}_3}{H^\al}^{2}.}
Note that this estimate is available as long as $\frac{3}{2}\ge \al >\frac{1}{2}$.
From this we have
\eqq{\frac{d}{dt}\norm{\bbar{U}_3(t)}{H^1}^2+\nu \norm{\bbar{U}_3(t)}{H^{1+\al}}^2\le C\nu ^{-\frac{3-2\al}{2\al -1}}\norm{\bbar{U}^h(t)}{H^{1}}^{\frac{2}{2\al -1}}\norm{\bbar{U}_3(t)}{H^\al}^{2},\qquad t>0.}
Integrating both sides in $t$ and applying \eqref{apri-Ubarh}, \eqref{apri-Ubar3-0}, we obtain that
\eq{apri-Ubar3}{&\norm{\bbar{U}_3(t)}{H^1}^2+\nu \int _0^t\norm{\bbar{U}_3(t')}{H^{1+\al}}^2\,dt'\\
&\le E^2+C\nu ^{-\frac{3-2\al}{2\al -1}}\Big( \sup _{0<t'<t}\norm{\bbar{U}^h(t')}{H^{1}}^{\frac{2}{2\al -1}}\Big) \int _0^t\norm{\bbar{U}_3(t')}{H^\al}^{2}\,dt'\\
&\le E^2+C\nu ^{-\frac{2}{2\al -1}}E^{\frac{2}{2\al -1}}M^2,\qquad t>0.}

\subsection{Non-trivial resonance part \eqref{FLosc}}

By the $L^2$ energy estimate with Lemma \ref{lem:BR2}, we immediately have
\eq{apri-Uosc1}{\norm{U\osc (t)}{L^2}^2+2\nu \int _0^t\norm{U\osc (t')}{H^\al}^2\,dt'\le \tnorm{u_{0,\mathrm{osc}}}{L^2}^2\le M^2,\qquad t>0.}

The $H^1$ bound will be obtained from the following lemma:
\begin{lem}\label{lem:combi}
For any $\e >0$ there exists $C>0$ such that for any real-valued, divergence-free and mean-zero vector field $a$, we have
\eqq{\big| \LR{B_R(a\osc ,a\osc )\,,\,a\osc}_{H^1}\big| \le C\tnorm{a\osc}{H^1}^2\tnorm{a\osc}{H^{\frac{3}{2}+\e}}.}
\end{lem}

This estimate, which is the most important piece in the proof of our result 
(we will give its proof in the next section),
improves in the case of regular (or rational) periodic domains the previous one proved in \cite[Theorem~3.1]{BMN2}.
The relation to the results of Babin \mbox{et al.} \cite{BMN1,BMN2} will be discussed in detail in the following subsection.

By Lemmas~\ref{lem:BR2} and \ref{lem:combi}, we proceed the $H^1$ energy estimate as
\eqq{\frac{d}{dt}\norm{U\osc (t)}{H^1}^2+2\nu \norm{U\osc (t)}{H^{1+\al}}^2\le C(\e )\norm{U\osc (t)}{H^1}^2\norm{U\osc (t)}{H^{\frac{3}{2}+\e}}.}
Let $\al \in (\frac{3}{4},1]$, and choose $\e >0$ so that $2\al >\frac{3}{2}+\e$.
By interpolation and Young's inequality, 
\eqq{&C(\e )\norm{U\osc}{H^1}^2\norm{U\osc}{H^{\frac{3}{2}+\e}}\le C(\e )\norm{U\osc}{H^{1+\al}}^{\frac{7}{2}-2\al +\e}\norm{U\osc}{L^2}\norm{U\osc}{H^\al}^{2\al -\frac{3}{2}-\e}\\
&\le \nu \norm{U\osc}{H^{1+\al}}^2+C(\e ,\al )\nu ^{-\frac{7-4\al +2\e}{4\al -3-2\e}}\norm{U\osc}{L^2}^{\frac{4}{4\al -3-2\e}}\norm{U\osc}{H^\al}^2,}
and hence,
\eqq{&\frac{d}{dt}\norm{U\osc (t)}{H^1}^2+\nu \norm{U\osc (t)}{H^{1+\al}}^2\\
&\quad \le C(\e ,\al )\nu ^{-\frac{7-4\al +2\e}{4\al -3-2\e}}\norm{U\osc (t)}{L^2}^{\frac{4}{4\al -3-2\e}}\norm{U\osc (t)}{H^\al}^2, \qquad t>0.}
Integrating both sides in $t$ and applying \eqref{apri-Uosc1}, we obtain that
\eq{apri-Uosc2}{&\norm{U\osc (t)}{H^1}^2+\nu \int _0^t\norm{U\osc (t')}{H^{1+\al}}^2\,dt'\\
&\le E^2+C(\e ,\al )\nu ^{-\frac{7-4\al +2\e}{4\al -3-2\e}}\Big( \sup _{0<t'<t}\norm{U\osc (t')}{L^2}^{\frac{4}{4\al -3-2\e}}\Big) \int _0^t\norm{U\osc (t')}{H^\al}^2\,dt'\\
&\le E^2+C(\e ,\al )\nu ^{-\frac{4}{4\al -(3+2\e )}}M^{2+\frac{4}{4\al -(3+2\e )}},\qquad t>0.}

Combining \eqref{apri-Ubarh}, \eqref{apri-Ubar3} and \eqref{apri-Uosc2}, we obtain a global $H^1$-a priori estimate on the solution $U(t)$ of the limit equation \eqref{limit} as
\eq{apri-U}{&\norm{U(t)}{H^1}^2+\nu \int _0^t\norm{U(t')}{H^{1+\al}}^2\,dt'\\
&\hx \le CE^2+C\nu ^{-\frac{2}{2\al -1}}M^2E^{\frac{2}{2\al -1}}+C(\e ,\al )\nu ^{-\frac{4}{4\al -(3+2\e )}}M^{2+\frac{4}{4\al -(3+2\e )}}\\
&\hx \lec _{\al ,\e}E^2\Big\{ 1+\Big( \frac{E}{\nu}\Big) ^{\frac{4}{4\al -(3+2\e )}}\Big\} ,\qquad t>0,}
where $0<\e <2\al -\frac{3}{2}$ and we have used $M\le E$, $0\le \frac{2}{2\al -1}\le \frac{4}{4\al -(3+2\e )}$.
Note that the last line of \eqref{apri-U} is constant in $t$.
This immediately implies that the solution $U(t)$ is bounded in $H^1$.
This is enough to show the existence of global regular solutions $U(t)$ of \eqref{limit}.
Moreover, \eqref{apri-U} also means that the $H^{1+\al}$ norm of $U(t)$ will eventually become arbitrarily small.
By the result of small-data global existence (similar to Theorem~\ref{prop:lwp}), we see that, for any $s\ge 1$, the $H^s$ norm of the solution $U(t)$ decays exponentially after some time.

We have thus established Proposition~\ref{prop:gwp-FL}, up to the proof of the key Lemma~\ref{lem:combi}.

\subsection{Remarks}\label{subsec:BMN}

Lemma~\ref{lem:combi} should be compared to the previous one by Babin \mbox{et al.} (\cite[Theorem~3.1]{BMN2}).
Let us recall some results in \cite{BMN1,BMN2}.

Let $a_2,a_3>0$ be positive numbers and consider the problem on the torus $\mathbf{T}_a^3:=[0,2\pi )\times [0,2\pi a_2)\times [0,2\pi a_3)$.
(We may always assume the period in the $x_1$ direction to be equal to $2\pi$ by rescaling the torus.)
The Fourier series is defined by
\eqq{u(x)=\sum _{n\in \Bo{Z}^3}\hhat{u}(n)e^{i\check{n}\cdot x},\quad \hhat{u}(n):=\frac{1}{(2\pi )^{3}a_2a_3}\int _{\mathbf{T}_a^3}u(x)e^{-i\check{n}\cdot x}\,dx,}
where $\check{n}=(n_1,\frac{n_2}{a_2},\frac{n_3}{a_3})$, and Sobolev spaces $H^s(\mathbf{T}_a^3)$ is defined in a natural way.
The eigenvalues of the matrix $\hhat{\HP}(n)J\hhat{\HP}(n)$ are $\pm i\check{n}_3/|\check{n}|$, and the nontrivial resonance condition can be written as
\eqq{\exists \sgm \in \{ \pm \} ^3;\qquad \sgm _1\frac{k_3}{|\check{k}|}+\sgm _2\frac{m_3}{|\check{m}|}=\sgm _3\frac{n_3}{|\check{n}|}}
with $k_3m_3n_3\neq 0$ and the convolution condition 
\eq{BMN-0}{k+m=n.}
This is equivalent to 
\eqq{0&=\prod _{\sgm _1,\sgm _2\in \{ \pm \}}\Big( \sgm _1\frac{k_3}{|\check{k}|}+\sgm _2\frac{m_3}{|\check{m}|}-\frac{n_3}{|\check{n}|}\Big) \\
&=\Big( \frac{k_3^2}{|\check{k}|^2}+\frac{m_3^2}{|\check{m}|^2}-\frac{n_3^2}{|\check{n}|^2}\Big) ^2-4\frac{k_3^2m_3^2}{|\check{k}|^2|\check{m}|^2},
}
or
\eq{BMN-1}{P(k,m,n; \theta_2, \theta_3 )=0,}
where we set $\th _2:=a_2^{-2}$, $\th _3:=a_3^{-2}$ and
\eqq{P(k,m,n; \theta_2, \theta_3 ):=\big( k_3^2|\check{m}|^2|\check{n}|^2+m_3^2|\check{k}|^2|\check{n}|^2-n_3^2|\check{k}|^2|\check{m}|^2\big) ^2-4k_3^2m_3^2|\check{k}|^2|\check{m}|^2|\check{n}|^4.}
Since
\eqq{|\check{n}|^2=n_1^2+\th _2n_2^2+\th _3n_3^2}
and similarly for $k,m$, $P(k,m,n;\theta_2 ,\theta_3 )$ is a polynomial of degree $4$ in $\th _2,\th _3$ and the coefficient of $\th _3^4$ ($=-3k_3^4m_3^4n_3^4$) does not vanish whenever $k_3m_3n_3\neq 0$.

On the other hand, \eqref{BMN-1} determines algebraic curves $\Gamma (k,m,n)$ in the $(\th _2,\th _3)$-plane parameterized by integer vectors $k,m,n$.
We see that the equation \eqref{BMN-1} (\mbox{i.e.} the curve $\Gamma(k,m,n)$) with the convolution condition \eqref{BMN-0} is invariant under dilations, reflections and permutations:
\eqs{(k,m,n)\mapsto (\la k,\la m,\la n),\qquad \la \in \R \setminus \{ 0\} ;\\
(k,m,n)\mapsto (Rk,Rm,Rn),\quad R\in \Big\{ \text{Id},\,\Big( \mat{-1 &0 &0\\ 0 &1 &0\\ 0 & 0 &1}\Big) ,\,\Big( \mat{1 &0 &0\\ 0 &-1 &0\\ 0 & 0 &1}\Big) ,\,\Big( \mat{-1 &0 &0\\ 0 &-1 &0\\ 0 & 0 &1}\Big) \Big\} ;\\
(k,m,n)\mapsto (S(k),S(m),-S(-n)),\quad \text{$S$: any permutation of $\{ k,m,-n\}$.}}

Hence, if we write $L_n$ to denote the straight line in the Fourier space through the origin and $n$, then the curve $\Gamma (k,m,n)$ depends only on (unordered) triplets $\{ L_k,L_m,L_n\}$ and does not depend on their (simultaneous) reflections.

It was shown in \cite[Section~4]{BMN2} that if we assume $k+m=n$, $k_3m_3n_3\neq 0$ and that the curve $\Gamma (k,m,n)$ intersects with the first quadrant of the $(\th _2,\th _3)$-plane, then the curve is represented as the graph of a function $\th _3=\phi _3(\th _2)$ on the first quadrant, and moreover, 
\begin{itemize}
\item if $(k _3m_2-k_2m_3)(k_1m_3-k_3m_1)(k_1m_2-k_2m_1)=0$, then the curve is reduced to a straight line;
\item otherwise, the curve is irreducible.
In this case, the coincidence of two such curves $\Gamma (k,m,n)=\Gamma (k',m',n')$ implies the coincidence of the sets
\eqq{&\Big\{ \frac{k_1^2}{k_3^2},\frac{m_1^2}{m_3^2},\frac{n_1^2}{n_3^2}\Big\} =\Big\{ \frac{(k'_1)^2}{(k'_3)^2},\frac{(m'_1)^2}{(m'_3)^2},\frac{(n'_1)^2}{(n'_3)^2}\Big\} ,\\
&\Big\{ \frac{k_2^2}{k_3^2},\frac{m_2^2}{m_3^2},\frac{n_2^2}{n_3^2}\Big\} =\Big\{ \frac{(k'_2)^2}{(k'_3)^2},\frac{(m'_2)^2}{(m'_3)^2},\frac{(n'_2)^2}{(n'_3)^2}\Big\} .}
\end{itemize}
Based on these facts, the numbers $N_r,N_{ir}$ were defined for given $\th _2,\th _3$ as follows:
\eqs{N_r(\th _2,\th _3):=\# \Big\{ \Gamma (k,m,n)\,\Big| \,\begin{matrix}k_3m_3n_3\neq 0,\,k+m=n,\, (\th _2,\th _3)\in \Gamma (k,m,n)\\ \text{$\Gamma$: straight line}\end{matrix}\Big\} ,\\
N_{ir}(\th _2,\th _3):=\sup _{L}~\# \Big\{ \Gamma (k,m,n)\,\Big| \,\begin{matrix}k_3m_3n_3\neq 0,\,k+m=n,\, (\th _2,\th _3)\in \Gamma (k,m,n)\\ \text{$\Gamma$: irreducible curve s.t.}\,RL\in \{ L_k,L_m,L_n\} \end{matrix}\Big\} ,
}
where $L$ ranges over all the lines through the origin and $R$ denotes reflection symmetries.

Babin \mbox{et al.} \cite{BMN1,BMN2} studied the global regularity for \eqref{NSC} with $\al =1$ in general periodic domains and made a refined analysis for domains with $N_{ir}(\th _2,\th _3)<\infty$.
Their result on the estimate of non-trivial resonant part can be rewritten with our notations as follows:
\begin{enumerate}
\item $N_r=N_{ir}=0$ holds for almost all $(\th _2,\th _3 )$.
In this case, non-trivial resonances do not occur; namely, $\LR{B_R(a\osc ,a\osc )\,,\,a\osc}_{H^1}\equiv 0$.
One has uniform-in-$\nu ^{-1}$ a priori bound
\eqq{\norm{U\osc (t)}{H^1}^2\le \norm{U\osc (0)}{H^1}^2}
for solutions $U\osc (t)$ of \eqref{FLosc} with $\al =1$, which even implies long time existence under fast rotation for inviscid flow, as shown in \cite{BMN1}.
\item In the case where $N_r=0$ and $0<N_{ir}<\infty$, non-trivial resonances do exist but are finitely many, \mbox{i.e.} ``$0$D like''.
One obtains an a priori bound
\eqq{\norm{U\osc (t)}{H^1}^2\le \norm{U\osc (0)}{H^1}^2+\frac{CN_{ir}}{\nu ^2}\norm{U\osc (0)}{L^2}^4.}
\item In the case where $0<N_r\le \infty$ and $0\le N_{ir}<\infty$, non-trivial resonance is ``$1$D like'', and the a priori bound obtained is 
\eqq{\norm{U\osc (t)}{H^1}^2\le \norm{U\osc (0)}{H^1}^2+\frac{CN_{ir}}{\nu ^2}\norm{U\osc (0)}{L^2}^4+\frac{C}{\nu ^4}\norm{U\osc (0)}{L^2}^6.}
\item In general, including the ``worst case'' of $N_{ir}=\infty$, one has the bound
\eqq{\norm{U\osc (t)}{H^1}^2\le \norm{U\osc (0)}{H^1}^2\exp \Big[ \frac{C}{\nu ^2}\norm{U\osc (0)}{L^2}^2\Big] .}
\end{enumerate}
In (ii) and (iii), one gets polynomial-in-$\nu ^{-1}$ a priori bound similar to \eqref{apri-Uosc2} shown above. In particular, one can also show global regularity for \eqref{NSC} with $\al$ less than $1$ on these domains (optimal range for $\al$ may depend on whether $N_r=0$ or not).
However, in the case (iv) the estimates in \cite{BMN2} are not sufficient to treat the fractional Laplacian with $\al <1$.

The general estimate in \cite[Theorem~3.1]{BMN2} was based on an observation that the total number of non-trivial resonant frequency triplets is ``2D like'', though the interactions are genuinely 3D.
This claim seems rather obvious (and actually was verified by a very elementary argument) because the resonance constraint represented by one nontrivial equality should reduce possibility by at least one dimension.
Now, it is also natural to expect that the non-trivial resonance is in fact much ``rarer'' event, since the resonance relation defines a surface of \emph{nonzero curvature} in frequency space.
As we will see in the proof of Lemma~\ref{lem:combi}, this kind of heuristics can be justified by a combinatorial argument in the case of regular or rational domains.
(Lemma~\ref{lem:combi} says that the non-trivial resonances are actually ``$(1+\e )$D like''.)
Such a combinatorial argument is a standard tool in the study of periodic nonlinear dispersive equations (see \mbox{e.g.} \cite{Bourgain93,CKSTT,K}), while it seems new in the context of equations of rotating fluids in a periodic domain.
We also note that it is not clear whether the ``$(1+\e )$D like'' estimate is optimal or not.

Finally, we claim that the regular (i.e. $\th _2=\th _3=1$) domains considered in this paper are in fact included in the $N_{ir}=\infty$ case.
However, it is hard in general to determine the precise values of $N_r$ and $N_{ir}$ for given $(\th _2,\th _3)$. 

\begin{lem}
$N_{ir}(1,1)=\I$.
\end{lem}
\begin{proof}
We will prove it by constructing infinitely many triplets $\{ (k_j,m_j,n_j)\} _{j\ge 1}$ satisfying $k_{j,3}m_{j,3}n_{j,3}\neq 0$, $k_j+m_j=n_j$, $n_j\in L_{(1,0,1)}$ and generating mutually different irreducible curves $\Gamma (k_j,m_j,n_j)$ in the $(\th _2,\th _3)$-plane passing through $(1,1)$.

Let us look for a triplet $(k,m,n)$ of the form
\eqq{k=(x,1,y),\quad m=(y,-1,x),\quad n=k+m=(x+y,0,x+y);\quad x,y\in \Bo{Z}.}
To ensure that $(1,1)\in \Gamma (k,m,n)$, we impose the following condition:
\eqq{\frac{k_3}{|k|}+\frac{m_3}{|m|}=\frac{n_3}{|n|}~~\text{or}~~-\frac{n_3}{|n|},\qquad \text{i.e.,}\qquad \frac{|x+y|}{\sqrt{x^2+y^2+1}}=\frac{1}{\sqrt{2}}.}
This is equivalent to
\eq{eq:Pell0}{x^2+4xy+y^2=1.}
Setting $X:=x+2y$, we get 
\eq{eq:Pell}{X^2-3y^2=1.}
This is one of Pell's equations known to have infinitely many integer solutions.
In fact, $(X,y)=(X_1,y_1):=(2,1)$ is a solution, and by the theory of Pell's equation, 
\eqq{(X_j,y_j)\in \Bo{N}^2\quad \text{defined by}\quad X_j+y_j\sqrt{3}=(X_1+y_1\sqrt{3})^j,\quad j=1,2,3,\dots}
are all solutions of \eqref{eq:Pell}.
Since
\eqq{X_{j+1}+y_{j+1}\sqrt{3}=(X_j+y_j\sqrt{3})(2+\sqrt{3})=(2X_j+3y_j)+(X_j+2y_j)\sqrt{3},}
the corresponding solutions $(x_j,y_j)$ of \eqref{eq:Pell0} with $X_j=x_j+2y_j$ satisfies the recurrence relations
\eqq{&y_{j+1}=X_j+2y_j=x_j+4y_j,\\
&x_{j+1}=X_{j+1}-2y_{j+1}=(2X_j+3y_j)-2(X_j+2y_j)=-y_j,}
with $(x_1,y_1)=(0,1)$.
Then, $(\ti{x}_j,\ti{y}_j):=(-1)^{j-1}(x_j,y_j)$ is also a solution of \eqref{eq:Pell0} satisfying
\eqq{\ti{x}_{j+1}=\ti{y}_j,\quad \ti{y}_{j+1}=-\ti{x}_j-4\ti{y}_j=-\ti{y}_{j-1}-4\ti{y}_j;\qquad (\ti{x}_1,\ti{y}_1)=(0,1).}
Therefore, we have $(\ti{x}_j,\ti{y}_j)=(a_{j-1},a_j)$ with the sequence $\{ a_j\} _{j\ge 0}$ defined by
\eq{def:aj}{a_0=0,\quad a_1=1,\quad a_{j+2}+4a_{j+1}+a_j=0\quad (j\ge 0).}

So far, we have obtained a sequence of triplets $\{ (k_j,m_j,n_j)\} _{j\ge 1}$,
\eqq{k_j=(a_j,1,a_{j+1}),\quad m_j=(a_{j+1},-1,a_j),\quad n_j=(a_j+a_{j+1},0,a_j+a_{j+1}),}
for which $k_j+m_j=n_j$, $n_j\in L_{(1,0,1)}$, and by the above construction of $\{ a_j\}$, the curve $\Gamma (k_j,m_j,n_j)$ passes through $(1,1)$.
(Given the sequence $\{ a_j\}$ defined by \eqref{def:aj}, one can also show directly without using the theory of Pell's equation that $(x,y)=(a_j,a_{j+1})$ satisfies \eqref{eq:Pell0}, and hence $(1,1)\in \Gamma (k_j,m_j,n_j)$, by an induction on $j$.)

It remains to check $k_{j,3}m_{j,3}n_{j,3}\neq 0$ and that $\{ \Gamma (k_j,m_j,n_j)\} _{j\ge 1}$ are mutually different irreducible curves.
By \eqref{def:aj}, we see that $|a_{j+1}|\ge 3|a_j|+1$ for any $j\ge 0$; in fact,
\eqq{|a_{j+1}|-3|a_{j}|&=|4a_{j}+a_{j-1}|-3|a_{j}|\ge (|a_{j}|-3|a_{j-1}|)+2|a_{j-1}|\\
&\ge |a_{j}|-3|a_{j-1}|\ge ~\cdots ~\ge |a_1|-3|a_0|=1.}
In particular, it holds that
\eqq{|a_j|\neq |a_{j'}|\quad (j\neq j');\qquad a_j\neq 0\quad (j\ge 1).}
This ensures that $k_{j,3}m_{j,3}n_{j,3}\neq 0$. 
Moreover, we have
\eqq{&(k _{j,3}m_{j,2}-k_{j,2}m_{j,3})(k_{j,1}m_{j,3}-k_{j,3}m_{j,1})(k_{j,1}m_{j,2}-k_{j,2}m_{j,1})\\
&=(a_j+a_{j+1})^3(a_j-a_{j+1})\neq 0,}
which shows that $\Gamma (k_j,m_j,n_j)$ is irreducible.
Finally, the sets
\eqq{\Big\{ \frac{k_{j,2}^2}{k_{j,3}^2},\,\frac{m_{j,2}^2}{m_{j,3}^2},\,\frac{n_{j,2}^2}{n_{j,3}^2}\Big\} =\{ a_{j+1}^{-2},\,a_j^{-2},\,0\} ,\qquad j\ge 1}
are mutually different, and so are the curves $\Gamma (k_j,m_j,n_j)$.
\end{proof}


\bigskip
\section{The key estimate on the non-trivial resonant part}\label{non-trivial resonant part}

Here, we shall give a proof of Lemma~\ref{lem:combi}.
Recall $\omega_{nkm}^\sigma=\sigma_1\frac{k_1}{|k|}+\sigma_2\frac{m_3}{|m|}+\sigma_3\frac{n_3}{|n|}$.
Define the set of non-trivial resonant frequencies $K^*\subset (\mathbb{Z}^3)^3$ as 
\[ K^*:=\big\{ (n,k,m)\in (\mathbb{Z}^3)^3\,\big| \,k+m=n,\,k_3m_3n_3\neq 0,\,\omega ^\sigma _{nkm}=0~\text{for some $\sigma \in \{ \pm \} ^3$}\big\} .\]
We also use the notation $\mathbb{Z}^3_*:=\Shugo{n\in \mathbb{Z}^3}{n_3\neq 0}$.

The following lemma is crucial in the proof:
\begin{lem}\label{lem:c2}
Let $L\ge 1$.
Then, for any $\varepsilon >0$ we have
\begin{equation*}
\sup _{n\in \mathbb{Z}_*^3}\# \big\{ (k,m)\in (\mathbb{Z}^3_*)^2 \,\big| \,(n,k,m)\in K^*,\,|k|\le L \big\} \le CL^{1+\varepsilon},
\end{equation*}
where the constant $C>0$ depends only on $\varepsilon$.
\end{lem}

\begin{rem}
In \cite{BMN2, CDGG}, they used the following estimate instead of the above:
\begin{equation*}
\sup _{n\in \mathbb{Z}_*^3}\# \big\{ (k,m)\in (\mathbb{Z}^3_*)^2 \,\big| \,(n,k,m)\in K^*,\,|k|\le L \big\} \le CL^2.
\end{equation*}
This estimate follows easily from the fact that the resonant constraint $\om ^\sgm _{nk(n-k)}=0$ determines an algebraic equation in $k_3$ of order $8$ for each fixed $n$ and $(k_1,k_2)$.
In particular, this estimate requires no combinatorial argument, and hence it holds for any periodic domains. 
On the other hand, the following proof of Lemma~\ref{lem:c2} is available only for regular (or rational) domains.
\end{rem}

\begin{proof}[Proof of Lemma~\ref{lem:c2}]
We rely on the well-known lemma in elementary number theory:
\begin{lem}[divisor bound, \mbox{cf.} Theorems 278 and 315 in \cite{HW}]\label{lem:c1}
For any $\varepsilon >0$ there exists $C>0$ such that the following estimates hold for any positive integer $N$.
\begin{enumerate}
\item $\# \{ \text{\emph{divisors of} }N \} \le CN^\varepsilon$.
\item $\# \{ (x,y)\in \mathbb{Z}^2 \,|\, x^2+y^2=N \} \le CN^\varepsilon$.
\end{enumerate}
\end{lem}

We focus on the case $\sigma =(+,+,+)$; a similar proof applies for other cases.

For given $n,k,m\in \mathbb{Z}_*^3$, positive integers $\nu ,\kappa ,\mu ,d_n,d_k,d_m$ are uniquely determined so that
\[ |n|=\nu \sqrt{d_n},\quad  |k|=\kappa \sqrt{d_k},\quad  |m|=\mu \sqrt{d_m},\quad \text{$d_n,d_k,d_m$ : square-free.}\]
We first see that $d_n=d_k=d_m$ if $\omega ^\sigma _{nkm}=0$.
In fact, we have
\[ \frac{n_3^2}{|n|^2}-\frac{2n_3k_3}{|n||k|}+\frac{k_3^2}{|k|^2}=\frac{m_3^2}{|m|^2},\]
hence $|n||k|=\nu \kappa \sqrt{d_nd_k}$ must be in $\mathbb{Q}$, which means $d_n=d_k$ since both $d_n$ and $d_k$ are square-free.
Similarly we have $d_n=d_m$.
Therefore, we may write uniquely as
\[  |n|=\nu \sqrt{d},\quad |k|=\kappa \sqrt{d},\quad |m|=\mu \sqrt{d},\quad \text{$d$ : square-free.}\]

Now, we fix $n\in \mathbb{Z}_*^3$ and count the number of $k\in \mathbb{Z}_*^3$ such that $n_3\neq k_3$, $\omega _{nk(n-k)}^\sigma =0$ and $|k|\le L$.
(Note that $\nu$, $d$ are determined once $n$ is fixed.)
First, there are at most $2L$ choices for $k_3$, since $|k|\le L$.

Next we fix $k_3$, so that $n_3-k_3$ is also fixed.
We shall prove that there are at most $O(L^\varepsilon )$ choices for $\kappa$.
Before proving it, we note that there are at most $O(L^{2\varepsilon} )$ choices for $(k_1,k_2)$ after fixing $\kappa$, because $k_1^2+k_2^2=|k|^2-k_3^2=\kappa ^2d-k_3^2=:N$ is a fixed positive integer and we can apply Lemma~\ref{lem:c1} (ii), noticing $N\le |k|^2\le L^2$.
These estimates imply the desired bound on the number of $k$'s.
More precisely, we just multiply all possibilities; $O(L)$ for $k_3$, $O(L^\e )$ 
for $\kappa$ and $O(L^{2\e})$ for $(k_1,k_2)$.

Now we estimate the total number of possible $\kappa$'s for fixed $n$ and $k_3$, considering the following three cases separately.

(I) $|n|\lesssim L^6$: We see that
\begin{align*}
\omega ^\sigma _{nk(n-k)}=0\quad &\Longleftrightarrow \quad \frac{k_3}{\kappa}+\frac{n_3-k_3}{\mu}=\frac{n_3}{\nu}\\
&\Longleftrightarrow \quad \big( n_3\kappa -k_3\nu \big) \big( n_3\mu -(n_3-k_3)\nu \big) =k_3(n_3-k_3)\nu ^2.
\end{align*}
Therefore, $n_3\kappa -k_3\nu \in \mathbb{Z}$ divides the fixed integer $k_3(n_3-k_3)\nu ^2$ of size $O(L^{1+6+6\cdot 2})$.
By Lemma~\ref{lem:c1} (i), there are at most $O(L^\varepsilon )$ choices for $n_3\kappa -k_3\nu \in \mathbb{Z}$.
This implies that there are at most $O(L^\e )$ possibilities for $\kappa$, because $n_3,k_3,\nu$ are all already determined.

(II) $|n|\gg L^6$, $|n_3|\lesssim |n|^{1/2}$: We show that this case does not occur.
In fact, it holds that $|n-k|\sim |n|$ and $|k|\le L\ll |n|^{1/2}$ in this case.
We have
\[ \frac{1}{L}\le \frac{1}{|k|}\le \left| \frac{k_3}{|k|}\right| \le \left| \frac{n_3}{|n|}\right| +\left| \frac{n_3-k_3}{|n-k|}\right| \lesssim \frac{|n|^{1/2}}{|n|}=\frac{1}{|n|^{1/2}},\]
which is not consistent with $|n|\gg L^6$.

(III) $|n|\gg L^6$, $|n_3|\gg |n|^{1/2}$: In this case we show that there are at most four choices for $\kappa$'s.
Suppose for contradiction that there are five possibilities for $\kappa$.
Since $(\kappa ,\mu )\in \mathbb{N}^2$ satisfies 
\[ \left( \kappa -\frac{k_3\nu}{n_3} \right) \left( \mu -\frac{(n_3-k_3)\nu}{n_3} \right) =\frac{k_3(n_3-k_3)\nu ^2}{n_3^2},\]
at least three different (non-collinear) points $P_j:=(\kappa _j,\mu _j)\in \mathbb{Z}^2$ ($j=1,2,3$) are on the same component of the fixed hyperbola (in this order):
\eqs{\{ (x,y)\in \mathbb{R}^2 : (x-a)(y-b)=M\} ,\\
a=\frac{k_3\nu}{n_3},\quad b=\frac{(n_3-k_3)\nu}{n_3},\quad M=\frac{k_3(n_3-k_3)\nu ^2}{n_3^2}.}

Now, an elementary calculation shows: For fixed $a,b, M\in \mathbb{R}$, the area $S$ of the region surrounded by (one component of) hyperbola $(x-a)(y-b)=M$ and a chord of length $\la$ is at most 
\begin{equation}\label{S}
S=O(\la ^3/\sqrt{|M|}),\quad\text{whenever}\quad \sqrt{|M|}\gg \la .
\end{equation}
To prove this, we may assume that $a=b=0$ and $M>0$ without loss of generality. 
Let $x_0>0$ and define $S(x_0)$ as the area of the region surrounded by the hyperbola and the segment between two points $(x_0,\frac{M}{x_0})$, $(x_0+\la ,\frac{M}{x_0+\la})$. 
By symmetry of hyperbolic curves on diagonal lines, $S$ is bounded by $\sup _{x_0\ge \sqrt{M}-\la}S(x_0)$ if $\sqrt{M}>\la$.
Now, for $x_0\ge \sqrt{M}-\la$, noticing that $\eta:=\frac{\la}{x_0}\ll 1$ if $\sqrt{M}\gg \la$, we have
\begin{eqnarray*}
S(x_0)&=&
\frac{\la}{2}\left(\frac{M}{x_0}+\frac{M}{x_0+\la}\right) -\int_{x_0}^{x_0+\la}\frac{M}{x}dx\\
&=&
M\left[\frac{\eta}{2}+\frac{\eta}{2(1+\eta)}-\log(1+\eta) \right]
\\
&=&
M\left[\frac{\eta}{2}+\frac{\eta}{2}(1-\eta+\eta^2+O(\eta^3))-(\eta-\frac{\eta^2}{2}+\frac{\eta^3}{3}+O(\eta^4 )) \right]\\
&=&
M(\frac{\eta^3}{6}+O(\eta^4)).
\end{eqnarray*}
By the above estimate, we have \eqref{S}.

In our case, $(\kappa ,\mu )$ is already confined to $[0,L/\sqrt{d}]\times [\nu -L/\sqrt{d},\nu +L/\sqrt{d}]$, so the length of the segment $P_1P_3$ is at most $\sqrt{5}L/\sqrt{d}$.
Since $|n_3|\gg L\ge |k_3|$ and $|n|\gg L^2$, we have
\[ \left| \frac{k_3(n_3-k_3)\nu ^2}{n_3^2}\right| =|k_3|\left| \frac{n_3-k_3}{n_3}\right| \left| \frac{\nu ^2}{n_3}\right| \gtrsim \frac{\nu ^2}{|n_3|}=\frac{|n|}{|n_3|}\frac{|n|}{d}\ge \frac{|n|}{d}\gg \left( \frac{L}{\sqrt{d}}\right) ^2.\]
Hence, we can apply \eqref{S} with $M\gec |n|/d$ and $\la \lec L/\sqrt{d}$ to show that the area of the triangle $P_1P_2P_3$ is bounded by 
\[ C\left( \frac{L}{\sqrt{d}}\right) ^3\left( \frac{d}{|n|}\right) ^{1/2}\lesssim \frac{L^3}{|n|^{1/2}}\ll 1, \]
where we have used the assumption $|n|\gg L^6$.
This is a contradiction, because the area of a non-degenerate lattice triangle is bounded from below by $\frac{1}{2}$.
Therefore, the case (III) has been proved.

This completes the proof of Lemma~\ref{lem:c2}.
\end{proof}

To show Lemma~\ref{lem:combi}, we use the following Sobolev estimate.

\begin{lem}\label{lem:RC}
Let $d\ge 1$ and $\rho \in [0,d]$.
Assume that a set $\La \subset \Shugo{(n,k,m)\in (\Bo{Z}^d)^3}{n+k+m=0}$ satisfies the following conditions:
\begin{itemize}
\item Symmetry: $(n,k,m)\in \La$ implies $(k,n,m),(n,m,k)\in \La$.
\item Dimension: There exists $C>0$ such that for any $L\ge 1$, 

$\sup\limits _{n\in \Bo{Z}^d}\# \Shugo{k\in \Bo{Z}^d}{(n,k,-n-k)\in \La,\,|k|\le L}\le CL^\rho$.
\end{itemize}
Let $\alpha ,\beta ,\gamma \in \mathbb{R}$ satisfy one of the following:
\begin{enumerate}
\item $\alpha +\beta +\gamma \ge \max \{ \alpha ,\beta ,\gamma \}$ and $\alpha +\beta +\gamma >\frac{\rho}{2}$;
\item $\alpha +\beta +\gamma > \max \{ \alpha ,\beta ,\gamma \}$ and $\alpha +\beta +\gamma = \frac{\rho}{2}$.
\end{enumerate}
Then, the following estimate holds.
\begin{equation*}
\Big| \sum _{(n,k,m)\in \La}
\widehat{f}(k)\widehat{g}(m)\widehat{h}(n)\Big| \lesssim \| f\| _{H^\alpha}\| g\| _{H^\beta}\| h\| _{H^\gamma}.
\end{equation*}
\end{lem}

This lemma can be proved by a standard argument using the Littlewood-Paley decomposition (cf.~\cite[Lemma~3.1]{BMN2}, \cite[Lemma~6.2]{CDGG}).
For the sake of completeness, we will give a proof in Section~\ref{sec:appendix}.

We observe that for the set $K^*$ of non-trivial resonant frequencies, Lemma~\ref{lem:c2} shows that $\La :=\Shugo{(n,k,m)}{(-n,k,m)\in K^*}$ satisfies the conditions in Lemma~\ref{lem:RC} with any $\rho \in (1,3]$.
We also notice that
\eqq{\LR{B_R(a\osc ,a\osc )\,,\,a\osc}_{H^1}=\LR{B_R(\nabla a\osc ,a\osc )\,,\,\nabla a\osc}_{L^2},}
since $\LR{B_R(a\osc ,\nabla a\osc )\,,\,\nabla a\osc}_{L^2}=0$ by Lemma~\ref{lem:BR2}.
Now, Lemma~\ref{lem:combi} is deduced from Lemma~\ref{lem:RC} with $\al =\be =0$ and $\ga =\frac{1}{2}+\e$.

\begin{rem}\label{rem:rational}
Lemma~\ref{lem:c2} also holds for any \emph{rational} domains; $\mathbf{T}^3_a=[0,2\pi a_1)\times [0,2\pi a_2)\times [0,2\pi a_3)$ satisfying $a_2^2/a_1^2,a_3^2/a_1^2\in \Bo{Q}$.
In fact, we may assume that $b_i:=a_i^{-2}\in \Bo{N}$, $i=1,2,3$ by a scaling argument.
Then, the resonance condition $\om ^\sgm _{nkm}=0$ is replaced by
\eqq{\sgm _1\frac{k_3}{|\check{k}|}+\sgm _2\frac{m_3}{|\check{m}|}=\sgm _3\frac{n_3}{|\check{n}|}}
with $|\check{k}|^2:=b_1k_1^2+b_2k_2^2+b_3k_3^2$.
Since $|\check{k}|^2\in \Bo{N}$ for any $k\in \Bo{Z}^3$, most of the above argument is applicable, except that Lemma~\ref{lem:c1}~(ii) should be modified as
\eqq{\# \{ (x,y)\in \mathbb{Z}^2 \,|\, b_1x^2+b_2y^2=N \} \le C(\e ,b_1,b_2)N^\varepsilon .}
This is actually true for any $b_1,b_2>0$ by the result of Bombieri and Pila~\cite[Theorem~3]{BP89}.
Once we have the key estimate (Lemma~\ref{lem:c2}), we can show the main result (Theorem~\ref{main}) for rational domains by the same arguments with some trivial modifications.
\end{rem}


\bigskip
\section{Error estimate and conclusion}\label{sec:error}

In this section, we give a proof of the main theorem (Theorem \ref{main}).
Let $\al \in (\frac{3}{4},1]$, $u_0\in H^1$ be an arbitrarily large initial vector field which is real-valued, divergence-free and mean-zero, and let $E:=\tnorm{u_0}{H^1}$.
We focus on the case 
\eqq{P:=\nu ^{-1}E \gec 1.}
In fact, $P\ll 1$ corresponds to the small-data case, where from Theorem~\ref{prop:lwp} we have a unique global solution to \eqref{NSC} for any $\Om \in \R$.

The purpose of this section is to see how global smooth solutions of \eqref{FNS2} (and hence, of \eqref{NSC}) are constructed from those of the limit equation \eqref{limit} in the fast rotation case ($|\Om |\ge \Om _0\gg 1$), and how $\Om _0$ depends on the initial vector field.
In what follows, $C(\al )$ denotes any positive constant depending on $\al$ with $C(\al )\to \I$ as $\al \downarrow \frac{3}{4}$, while $C$ denotes any absolute positive constant.

Theorem~\ref{main} follows once we have the same result for the equation \eqref{FNS2}.
This will be shown by estimating the $H^1$ distance between solution $v(t)$ of \eqref{FNS2} and the corresponding global-in-time solution $U(t)$ of \eqref{limit}.

By \eqref{apri-U}, we know that
\eq{apri-U2}{\tnorm{U(t)}{H^1}^2+\nu \int _0^t\tnorm{U(t')}{H^{1+\al}}^2\,dt'\le \ti{E}^2:=C(\al )E^2P^{C(\al )},\quad t\ge 0.}
Let $\ti{T}_L$ be the local existence time of the $H^1$ solution to \eqref{FNS2} of size $2\ti{E}$.

We shall prove the following by induction:
If we define $\Om _0=\Om _0(\al ,\nu ,E)>0$ as \eqref{def:Om0}, then for $n=1,2,\cdots$, the solution $v(t)$ to \eqref{FNS2} with $|\Om |\ge \Om _0$ exists on $[0,n\ti{T}_L]$ and satisfies
\eq{apri-v}{\tnorm{v(t)}{H^1}^2+\nu \int _0^t\tnorm{v(t')}{H^{1+\al}}^2\,dt'\le (2\ti{E})^2}
for $t\in [0,n\ti{T}_L]$.
For $n=1$ this follows from Theorem~\ref{prop:lwp}, so we assume it for some $n$.
Since $\tnorm{v(n\ti{T}_L)}{H^1}\le 2\ti{E}$, by Theorem~\ref{prop:lwp} again, $v$ extends up to $t=(n+1)\ti{T}_L$  and we have a larger bound:
\eq{apri-v2}{\tnorm{v(t)}{H^1}^2+\nu \int _0^t\tnorm{v(t')}{H^{1+\al}}^2\,dt'\le L^2:=C(\al )(2\ti{E})^2,\quad t\in [0,(n+1)\ti{T}_L].}
It then suffices to show \eqref{apri-v} on $[0,(n+1)\ti{T}_L]$ from \eqref{apri-U2} and \eqref{apri-v2}.


Let us first prepare some useful estimates.
\begin{lem}
Let $\e >0$.
We have
\begin{gather}
\label{est:BR} \big| \LR{B_R(f,g)\,,\,h}_{H^1}\big| \lec _\e \tnorm{f}{H^1}\tnorm{g}{H^{\frac{3}{2}+\e}}\tnorm{h}{H^{\frac{3}{2}}},\\
\label{est:BNR} \big| \LR{B_{N\!R}(\Om t;f,g)\,,\,h}_{H^1}\big| \lec \tnorm{f}{H^1}\tnorm{g}{H^{\frac{7}{4}}}\tnorm{h}{H^{\frac{7}{4}}}.
\end{gather}
\end{lem}

\begin{proof}
We consider the sets
\eqq{\La _R&:=\Shugo{(n,k,m)\in (\Bo{Z}^3\setminus \{ 0\} )^3}{n+k+m=0,\, \om _{nkm}^\sgm =0~\text{for some $\sgm \in \{ \pm \} ^3$}},\\
\La _{N\!R}&:=\Shugo{(n,k,m)\in (\Bo{Z}^3\setminus \{ 0\} )^3}{n+k+m=0,\, \om _{nkm}^\sgm \neq 0~\text{for any $\sgm \in \{ \pm \} ^3$}}.}
Clearly, $\La _{N\!R}$ satisfies the hypotheses of Lemma~\ref{lem:RC} with $\rho =3$.
We recall that $(n,k,m)\in \La _R$ implies $P(k,-k-n,n;1,1)=0$, where $P$ is the polynomial defined in \eqref{BMN-1} with the coefficient of $k_3^8$ being $-(4|n|^2-n_3^2)n_3^2$.
This shows that if $n_3\neq 0$, for $n$, $k_1$ and $k_2$ fixed, there are at most $8$ possibilities for $k_3$.
On the other hand, when $n_3=0$, the resonant condition implies that $|k|=|-n-k|$, so that $k$ must be on the hyperplane passing through $-n/2$ and orthogonal to $-n$.
Hence, $\La_R$ satisfies the hypotheses of Lemma~\ref{lem:RC} with $\rho =2$.

We apply Lemma~\ref{lem:RC} with $\La =\La _R$ and $(\al ,\be ,\ga )=(\frac{1}{2},\frac{1}{2}+\e ,0),(1,\e ,0)$ to have
\eqq{
&\sum _{(n,k,m)\in \La _R}
|\widehat{f}(k)|\cdot |m||\widehat{g}(m)|\cdot |n|^2|\widehat{h}(n)| \\
&\le \sum _{(n,k,m)\in \La _R}
\Big( |k|^{\frac{1}{2}}|\widehat{f}(k)|\cdot |m||\widehat{g}(m)|+|\widehat{f}(k)|\cdot |m|^{\frac{3}{2}}|\widehat{g}(m)|\Big) |n|^{\frac{3}{2}}|\widehat{h}(n)|\\
&\lec \| f\| _{H^1}\| g\| _{H^{\frac{3}{2}+\e}}\| h\| _{H^{\frac{3}{2}}}.
}
This estimate implies \eqref{est:BR}.
The estimate \eqref{est:BNR} follows from Lemma~\ref{lem:RC} with $\La =\La _{N\!R}$ and $(\al ,\be ,\ga )=(\frac{3}{4},\frac{3}{4},0),(1,\frac{1}{2},0)$ as
\eqq{
&\sum _{(n,k,m)\in \La _{N\!R}}
|\widehat{f}(k)|\cdot |m||\widehat{g}(m)|\cdot |n|^2|\widehat{h}(n)| \\
&\le \sum _{(n,k,m)\in \La _{N\!R}}
\Big( |k|^{\frac{1}{4}}|\widehat{f}(k)|\cdot |m||\widehat{g}(m)|+|\widehat{f}(k)|\cdot |m|^{\frac{5}{4}}|\widehat{g}(m)|\Big) |n|^{\frac{7}{4}}|\widehat{h}(n)|\\
&\lec \| f\| _{H^1}\| g\| _{H^{\frac{7}{4}}}\| h\| _{H^{\frac{7}{4}}}.\qedhere
}
\end{proof}

Now, we estimate the difference $w(t):=v(t)-U(t)$, which is smooth for $0<t\le (n+1)\ti{T}_L$ and satisfies
\IVP{eq:w}{&\p _tw+\nu (-\Delta )^\al w+B_R(w,v)+B_R(U,w)+B_{N\!R}(\Om t; v,v)=0,\quad t>0,\\
&w\big| _{t=0}=0.
}
In view of \eqref{apri-U2}, we will obtain \eqref{apri-v} on $[0,(n+1)\ti{T}_L]$ once we show that
\eq{apri-w}{\tnorm{w(t)}{H^1}^2+\nu \int _0^t\tnorm{w(t')}{H^{1+\al}}^2\,dt'\le \ti{E}^2,\qquad t\in [0,(n+1)\ti{T}_L].}

By the $H^1$ energy argument on \eqref{eq:w} and the estimate \eqref{est:BR}, together with interpolation, we have
\eqq{&\frac{d}{dt}\tnorm{w(t)}{H^1}^2+2\nu \tnorm{w(t)}{H^{1+\al}}^2\\
&\le C\tnorm{v(t)}{H^{1+\al}}\tnorm{w(t)}{H^1}\tnorm{w(t)}{H^{1+\al}}+C\tnorm{U(t)}{H^1}\tnorm{w(t)}{H^1}^{\frac{1}{2}}\tnorm{w(t)}{H^{1+\al}}^{\frac{3}{2}}\\
&\hx -2\LR{B_{N\!R}(\Om t; v(t),v(t))\,,\,w(t)}_{H^1}.}
Using Young's inequality, we have
\eqq{\frac{d}{dt}\tnorm{w(t)}{H^1}^2+\nu \tnorm{w(t)}{H^{1+\al}}^2&\le C\Big( \nu ^{-1}\tnorm{v(t)}{H^{1+\al}}^2+\nu ^{-3}\tnorm{U(t)}{H^1}^4\Big) \tnorm{w(t)}{H^1}^2\\
&\hx -2\LR{B_{N\!R}(\Om t; v(t),v(t))\,,\,w(t)}_{H^1},}
and hence, by \eqref{apri-U2},
\eq{apri-w2}{&\tnorm{w(t)}{H^1}^2+\nu \int _0^t\tnorm{w(t')}{H^{1+\al}}^2\,dt'\\
&\le C\int _0^t\Big( \nu ^{-1}\tnorm{v(t')}{H^{1+\al}}^2+\nu ^{-3}\ti{E}^2\tnorm{U(t')}{H^{1+\al}}^2\Big) \tnorm{w(t')}{H^1}^2\,dt'\\
&\hx -2\int _0^t\LR{B_{N\!R}(\Om t'; v(t'),v(t'))\,,\,w(t')}_{H^1}\,dt',\qquad t\in [0,(n+1)\ti{T}_L].}

To control the last integral, we claim the following:
\begin{lem}\label{lem:est-Om}
For given $\de >0$, there exists $\Om _0=\Om _0(\de , \al ,\nu ,E)>0$ such that if $|\Om |\ge \Om _0$, then we have
\eqq{&\Big| 2\int _0^t\LR{B_{N\!R}(\Om t'; v(t'),v(t'))\,,\,w(t')}_{H^1}\,dt'\Big| \\
&\le \de +\frac{1}{2}\Big( \tnorm{w(t)}{H^1}^2+\nu \int _0^t\tnorm{w(t')}{H^{1+\al}}^2\,dt'\Big) ,\qquad t\in [0,(n+1)\ti{T}_L].}
One can take $\Om _0$ as
\eqq{\Om _0=C(\al )P^{C(\al )}(\de ^{-1}\nu ^2)^{C(\al )}E.}
\end{lem}

Let us admit the above lemma and continue the proof.
From \eqref{apri-w2}, we have
\eqq{&\tnorm{w(t)}{H^1}^2+\nu \int _0^t\tnorm{w(t')}{H^{1+\al}}^2\,dt'\\
&\le 2\de +C\int _0^t\Big( \nu ^{-1}\tnorm{v(t')}{H^{1+\al}}^2+\nu ^{-3}\ti{E}^2\tnorm{U(t')}{H^{1+\al}}^2\Big) \tnorm{w(t')}{H^1}^2\,dt'}
for $t\in [0,(n+1)\ti{T}_L]$ if $|\Om |\ge \Om _0$, with $\de >0$ to be chosen later.
By the Gronwall inequality and \eqref{apri-U2}, \eqref{apri-v2},
\eqq{&\norm{w(t)}{H^1}^2+\nu \int _0^t\norm{w(t')}{H^{1+\al}}^2\,dt'\\
&\le 2\de \exp \Big[ C\int _0^t\Big( \nu ^{-1}\tnorm{v(t')}{H^{1+\al}}^2+\nu ^{-3}\ti{E}^2\tnorm{U(t')}{H^{1+\al}}^2\Big) \,dt'\Big] \\
&\le 2\de e^{C(\nu ^{-2}L^2+\nu ^{-4}\ti{E}^4)},\qquad t\in [0,(n+1)\ti{T}_L].}
We finally obtain the claimed estimate \eqref{apri-w} by choosing $\de$ as
\eqq{2\de e^{C(\nu ^{-2}L^2+\nu ^{-4}\ti{E}^4)} \le \ti{E}^2.}
Since $\ti{E},L\le C(\al )P^{C(\al )}E$, we can take $\Om _0$ as \eqref{def:Om0}.
Note that $\de$ does not depend on $n$, so that we can proceed to bigger and bigger $n$ without redefining $\Om _0$.

This concludes the proof of Theorem~\ref{main}, up to the proof of Lemma~\ref{lem:est-Om}.

\begin{proof}[Proof of Lemma~\ref{lem:est-Om}]
Let $N$ be a large positive number to be chosen later, and $P_{\le N}:=\F ^{-1}\chi _{\{ |n|\le N\}}\F$, $P_{>N}:=1-P_{\le N}$.
We see that
\eqq{&\LR{B_{N\!R}(\Om t; v,v)\,,\,w}_{H^1}-\LR{B_{N\!R}(\Om t; P_{\le N}v,P_{\le N}v)\,,\,w}_{H^1}\\
&=\LR{B_{N\!R}(\Om t; P_{>N}v,v)\,,\,w}_{H^1}+\LR{B_{N\!R}(\Om t; P_{\le N}v,P_{>N}v)\,,\,w}_{H^1}\\
&=\LR{B_{N\!R}(\Om t; P_{>N}v,P_{>N/2}v)\,,\,w}_{H^1}+\LR{B_{N\!R}(\Om t; P_{>N}v,P_{\le N/2}v)\,,\,P_{>N/2}w}_{H^1}\\
&\hx +\LR{B_{N\!R}(\Om t; P_{>N}v,P_{\le N/2}v)\,,\,P_{\le N/2}w}_{H^1}+\LR{B_{N\!R}(\Om t; P_{\le N}v,P_{>N}v)\,,\,w}_{H^1}.}
Note that the third term vanishes in the right-hand side of the last equality.
By the inequality
\[ \tnorm{P_{>N}f}{H^{\frac{7}{4}}}\le N^{-(\al -\frac{3}{4})}\tnorm{f}{H^{1+\al}}\]
and \eqref{est:BNR}, we then obtain 
\eqq{&\Big| \LR{B_{N\!R}(\Om t; v,v)\,,\,w}_{H^1}-\LR{B_{N\!R}(\Om t; P_{\le N}v,P_{\le N}v)\,,\,w}_{H^1}\Big| \\
&\le CN^{-(\al -\frac{3}{4})}\tnorm{v}{H^1}\tnorm{v}{H^{1+\al}}\tnorm{w}{H^{1+\al}}\\
&\le \frac{\nu}{8}\tnorm{w}{H^{1+\al}}^2+C\nu ^{-1}N^{-(2\al -\frac{3}{2})}\tnorm{v}{H^1}^2\tnorm{v}{H^{1+\al}}^2.}
Invoking \eqref{apri-v2}, we have
\eq{est:Om1}{&2\int _0^t\Big| \LR{B_{N\!R}(\Om t'; v(t'),v(t'))-B_{N\!R}(\Om t'; P_{\le N}v(t'),P_{\le N}v(t'))\,,\,w(t')}_{H^1}\Big| \,dt'\\
&\le \frac{\nu}{4}\int _0^t\tnorm{w(t')}{H^{1+\al}}^2\,dt'+C\nu ^{-2}L^4N^{-(2\al -\frac{3}{2})}.}


To estimate the low-frequency term, we first claim that:
\eq{claim:lowerom}{\inf \bigg\{ |\om ^\sgm _{nk(n-k)}|\,\bigg| \,\begin{matrix} \sgm \in \{ \pm \} ^3,\,n,k\in \Bo{Z}^3\setminus \{ 0\} ~\text{s.t.}~n\neq k,\\
\om ^\sgm _{nk(n-k)}\neq 0,\,|k|\le N,\,|n-k|\le N\end{matrix}\bigg\} \gec N^{-12}.}

To this end, we take $k,n\neq 0$ such that $|k|,|n-k|\le N$ and $n-k\neq 0$.
If there is a $\sgm \in \{ \pm \}^3$ such that $\om ^\sgm _{nk(n-k)}=0$, then we see that
\eqq{\om ^\sgm _{nk(n-k)}\in \shugo{ 0,\,\pm 2\frac{k_3}{|k|},\pm 2\frac{n_3-k_3}{|n-k|},\,\pm 2\frac{n_3}{|n|}}}
for any $\sgm \in \{ \pm \} ^3$.
In this case, $|\om ^\sgm _{nk(n-k)}|\gec N^{-1}$ unless $\om ^\sgm _{nk(n-k)}=0$.

We thus assume that $\om ^\sgm _{nk(n-k)}\neq 0$ for all $\sgm$.
In this case, by the identity
\eqq{&\prod _{\sgm _1,\sgm _2\in \{ \pm \}}\om ^{\sgm_1,\sgm _2,-}_{nk(n-k)}\\
&=\frac{k_3^4}{|k|^4}+\frac{(n_3-k_3)^4}{|n-k|^4}+\frac{n_3^4}{|n|^4}-2\Big( \frac{k_3^2}{|k|^2}\frac{(n_3-k_3)^2}{|n-k|^2} +\frac{n_3^2}{|n|^2}\frac{k_3^2}{|k|^2}+\frac{(n_3-k_3)^2}{|n-k|^2}\frac{n_3^2}{|n|^2}\Big) \\
&=\frac{\text{(non-zero integer)}}{|k|^4|n-k|^4|n|^4},}
the product of these four $\omega$'s has a lower bound $2^{-4}N^{-12}$. 
Since each of them has an upper bound $|\om ^\sgm _{nk(n-k)}|\le 3$, we have $|\om ^\sgm _{nk(n-k)}|\ge 3^{-3}2^{-4}N^{-12}$ for any $\sgm$.
Therefore, \eqref{claim:lowerom} has been proved.

By integration by parts in $t'$, we see that
\eqq{&\int _0^t\LR{B_{N\!R}(\Om t'; P_{\le N}v(t'),P_{\le N}v(t'))\,,\,w(t')}_{H^1}\,dt'\\
&=i\int _0^t\sum _{\sgm}\sum _{\mat{n,k\in \Bo{Z}^3\setminus \{ 0\} \\ \om ^\sgm _{nk(n-k)}\neq 0\\ |k|,|n-k|\le N}}e^{-i\Om t'\om ^\sgm _{nk(n-k)}}\big[ \hhat{v}^{\sgm _1}(t',k)\cdot (n-k)\big] \big[ \hhat{v}^{\sgm _2}(t',n-k)\cdot |n|^2\hhat{w}^{\sgm _3}(t',n)^*\big] \,dt'\\
&=\bigg[ \sum _{\sgm}\sum _{\mat{\om ^\sgm _{nk(n-k)}\neq 0\\ |k|,|n-k|\le N}}\frac{e^{-i\Om t'\om ^\sgm _{nk(n-k)}}}{-\Om \om ^\sgm _{nk(n-k)}}\big[ \hhat{v}^{\sgm _1}(t',k)\cdot (n-k)\big] \big[ \hhat{v}^{\sgm _2}(t',n-k)\cdot |n|^2\hhat{w}^{\sgm _3}(t',n)^*\big] \bigg] _0^t\\
&\hx +\int _0^t\sum _{\sgm}\sum _{\mat{\om ^\sgm _{nk(n-k)}\neq 0\\ |k|,|n-k|\le N}}\frac{e^{-i\Om t'\om ^\sgm _{nk(n-k)}}}{\Om \om ^\sgm _{nk(n-k)}}\big[ \p _{t'}\hhat{v}^{\sgm _1}(t',k)\cdot (n-k)\big] \big[ \hhat{v}^{\sgm _2}(t',n-k)\cdot |n|^2\hhat{w}^{\sgm _3}(t',n)^*\big] \,dt'\\
&\hx +\int _0^t\sum _{\sgm}\sum _{\mat{\om ^\sgm _{nk(n-k)}\neq 0\\ |k|,|n-k|\le N}}\frac{e^{-i\Om t'\om ^\sgm _{nk(n-k)}}}{\Om \om ^\sgm _{nk(n-k)}}\big[ \hhat{v}^{\sgm _1}(t',k)\cdot (n-k)\big] \big[ \p _{t'}\hhat{v}^{\sgm _2}(t',n-k)\cdot |n|^2\hhat{w}^{\sgm _3}(t',n)^*\big] \,dt'\\
&\hx +\int _0^t\sum _{\sgm}\sum _{\mat{\om ^\sgm _{nk(n-k)}\neq 0\\ |k|,|n-k|\le N}}\frac{e^{-i\Om t'\om ^\sgm _{nk(n-k)}}}{\Om \om ^\sgm _{nk(n-k)}}\big[ \hhat{v}^{\sgm _1}(t',k)\cdot (n-k)\big] \big[ \hhat{v}^{\sgm _2}(t',n-k)\cdot |n|^2\p _{t'}\hhat{w}^{\sgm _3}(t',n)^*\big] \,dt'.
}
We assume that $|\Om |$ is greater than some $\Om _0$ to be determined.
Invoking \eqref{claim:lowerom} and \eqref{est:BNR}, we have
\eqq{&\Big| \int _0^t\LR{B_{N\!R}(\Om t'; P_{\le N}v(t'),P_{\le N}v(t'))\,,\,w(t')}_{H^1}\,dt'\Big| \\
&\le \frac{CN^{12}}{\Om _0}\bigg[ \norm{P_{\le N}v(t)}{H^1}\norm{P_{\le N}v(t)}{H^{\frac{7}{4}}}\norm{P_{\le 2N}w(t)}{H^{\frac{7}{4}}}\\
&\hxx +\int _0^t\Big( \norm{P_{\le N}\p _{t'}v(t')}{H^1}\norm{P_{\le N}v(t')}{H^{\frac{7}{4}}}\norm{P_{\le 2N}w(t')}{H^{\frac{7}{4}}}\\
&\hxx\hxx +\norm{P_{\le N}v(t')}{H^1}\norm{P_{\le N}\p _{t'}v(t')}{H^{\frac{7}{4}}}\norm{P_{\le 2N}w(t')}{H^{\frac{7}{4}}}\\
&\hxx\hxx +\norm{P_{\le N}v(t')}{H^1}\norm{P_{\le N}v(t')}{H^{\frac{7}{4}}}\norm{P_{\le 2N}\p _{t'}w(t')}{H^{\frac{7}{4}}}\Big) dt'\bigg] \\
&\le \frac{CN^{12}}{\Om _0}\bigg[ N^{\frac{3}{2}}\tnorm{v(t)}{H^1}^2\tnorm{w(t)}{H^1}\\
&\hx +N^{\frac{7}{4}}\int _0^t\Big( \tnorm{\p _{t'}v(t')}{L^2}\tnorm{v(t')}{H^1}\tnorm{w(t')}{H^{\frac{7}{4}}}+\tnorm{v(t')}{H^1}\tnorm{v(t')}{H^{\frac{7}{4}}}\tnorm{\p _{t'}w(t')}{L^2}\Big) dt'\bigg] .}

Time derivatives of $v$ and $w$ can be estimated by using the equations and Lemma~\ref{lem:Sobolev} (assuming $\al \in (\frac{3}{4},1]$), as follows:
\eqq{\tnorm{\p _tv(t)}{L^2}&=\norm{\nu (-\Delta )^\al v+B(\Om t;v,v)}{L^2}\le \nu \tnorm{v}{H^{1+\al}}+C\tnorm{v}{H^1}\tnorm{v}{H^{1+\al}},\\
\tnorm{\p _tw(t)}{L^2}&=\norm{\nu (-\Delta )^\al w+B_R(w,v)+B_R(U,w)+B_{N\!R}(\Om t; v,v)}{L^2}\\
&\le \nu \tnorm{w}{H^{1+\al}}+C\big( \tnorm{v}{H^1}+\tnorm{U}{H^1}\big) \tnorm{w}{H^{1+\al}}+C\tnorm{v}{H^1}\tnorm{v}{H^{1+\al}}.}

By these estimates and Young's inequality with \eqref{apri-U2} and \eqref{apri-v2}, we obtain that
\eqq{&2\Big| \int _0^t\LR{B_{N\!R}(\Om t'; P_{\le N}v(t'),P_{\le N}v(t'))\,,\,w(t')}_{H^1}\,dt'\Big| \\
&\le \frac{CN^C}{\Om _0}\bigg[ \tnorm{v(t)}{H^1}^2\tnorm{w(t)}{H^1}\\
&\hx +\int _0^t\Big\{ \tnorm{w}{H^{1+\al}}\tnorm{v}{H^1}\big( \nu +\tnorm{U}{H^1}+\tnorm{v}{H^1}\big) \tnorm{v}{H^{1+\al}}+\tnorm{v}{H^1}^2\tnorm{v}{H^{1+\al}}^2\Big\} \,dt'\bigg] \\
&\le \frac{1}{4}\Big( \tnorm{w(t)}{H^1}^2+\nu \int _0^t\tnorm{w(t')}{H^{1+\al}}^2\,dt'\Big) +\frac{CN^C}{\Om _0}\int _0^t \tnorm{v}{H^1}^2\tnorm{v}{H^{1+\al}}^2\,dt'\\
&\hx +\frac{CN^C}{\Om _0^2}\Big( \tnorm{v(t)}{H^1}^4+\nu ^{-1}\int _0^t \tnorm{v}{H^1}^2\big( \nu ^2+\tnorm{U}{H^1}^2+\tnorm{v}{H^1}^2\big) \tnorm{v}{H^{1+\al}}^2\,dt' \Big) \\
&\le \frac{1}{4}\Big( \tnorm{w(t)}{H^1}^2+\nu \int _0^t\tnorm{w(t')}{H^{1+\al}}^2\,dt'\Big) +\frac{CN^CL^4}{\Om _0\nu}+\frac{CN^C}{\Om _0^2}\Big( 1+\frac{\ti{E}^2+L^2}{\nu ^2}\Big) L^4.
}
Combining it with \eqref{est:Om1}, we have
\eqq{&2\Big| \int _0^t\LR{B_{N\!R}(\Om t'; v(t'),v(t'))\,,\,w(t')}_{H^1}\,dt'\Big| \\
&\le \frac{1}{2}\Big( \tnorm{w(t)}{H^1}^2+\nu \int _0^t\tnorm{w(t')}{H^{1+\al}}^2\,dt'\Big) \\
&\hx +C\nu ^2\bigg[ \Big( \frac{L}{\nu}\Big) ^4N^{-(2\al -\frac{3}{2})}+\frac{N^C}{\Om _0/\nu}\Big( \frac{L}{\nu}\Big) ^4+\frac{N^C}{(\Om _0/\nu )^2}\Big\{ 1+\Big( \frac{\ti{E}+L}{\nu}\Big) ^2\Big\} \Big( \frac{L}{\nu}\Big) ^4\bigg] }
for $t\in [0,(n+1)\ti{T}_L]$.
Recalling that $\frac{\ti{E}}{\nu},\frac{L}{\nu}\le C(\al )P^{C(\al )}$, the above bound can be rewritten as
\eqq{\frac{1}{2}\Big( \tnorm{w(t)}{H^1}^2+\nu \int _0^t\tnorm{w(t')}{H^{1+\al}}^2\,dt'\Big) +\nu ^2\bigg[ N^{-(2\al -\frac{3}{2})}+\frac{N^C}{\Om _0/\nu}\bigg] C(\al )P^{C(\al )}}

Now, for given $\de$, we take $N$ as
\eqq{C(\al )P^{C(\al )}\nu ^2N^{-(2\al -\frac{3}{2})}\le \frac{\de}{2}\qquad \text{i.e.}\quad N\ge C(\al )P^{C(\al )}(\de ^{-1}\nu ^2)^{C(\al )}
}
and then take $\Om _0$ so that
\eqs{\frac{C(\al )P^{C(\al )}\nu ^2N^C}{\Om _0/\nu}\le \frac{\de}{2}\\
\text{i.e.}\quad \Om _0\ge C(\al )P^{C(\al )}\de ^{-1}\nu ^3(\de ^{-1}\nu ^2)^{C(\al )}N^C\ge C(\al )P^{C(\al )}(\de ^{-1}\nu ^2)^{C(\al )}E,}
concluding the proof of Lemma~\ref{lem:est-Om}.
\end{proof}


\bigskip
\section{Appendix}

\subsection{Sobolev estimates}\label{sec:appendix}

Here, we give a proof of Lemma~\ref{lem:RC}.
\begin{proof}[Proof of Lemma~\ref{lem:RC}]
By symmetry of $\La$, we may restrict the summation onto the frequencies satisfying $|k|\ge |m|\ge |n|$ in the left-hand side of the claimed estimate.

We define the dyadic set $\Sigma _j:=\{ k\in \mathbb{Z}^d:2^j\le |k|<2^{j+1}\}$ for $j=0,1,2,\dots$ and decompose mean-zero $f$ as $f=\sum _{j\ge 0}f_j$ with $\widehat{f}_j:=\widehat{f}\chi _{\Sigma _j}$, and similarly for $g$ and $h$.
Note that $\| f\|_{H^\alpha}=\big( \sum _{j\ge 0}\| f_j\|_{H^\alpha}^2\big) ^{1/2}$.
Since $n+k+m=0$ and $|k|\ge |m|\ge |n|$ implies $|k|\le 2|m|$, it holds that
\[ S:=\Big| \sum _{
\begin{smallmatrix}
(n,k,m)\in \La \\
|k|\ge |m|\ge |n|
\end{smallmatrix}
}
\widehat{f}(k)\widehat{g}(m)\widehat{h}(n)\Big| 
\le \sum _{j\ge 0}\sum _{j'=j,j+1}\sum _{0\le l\le j}\sum _{(n,k,m)\in \La}\big| \widehat{f}_{j'}(k)\widehat{g}_j(m)\widehat{h}_l(n)\big| .
\]

By the Cauchy-Schwarz inequality and the dimension hypothesis on $\La$, we have
\begin{align*}
&2^{\alpha j'+\beta j+\gamma l}\sum _{(n,k,m)\in \La}\big| \widehat{f}_{j'}(k)\widehat{g}_j(m)\widehat{h}_l(n)\big| \\
&\le \Big( \sum _{k\in \mathbb{Z}^3}2^{2\alpha j'}|\widehat{f}_{j'}(k)|^2\Big) ^{1/2}\Big( \sum _{k\in \mathbb{Z}^3}\Big( \sum _{\begin{smallmatrix} m,n\in \mathbb{Z}^3\\ (n,k,m)\in \La \end{smallmatrix}}2^{\beta j}|\widehat{g}_j(m)|\cdot 2^{\gamma l}|\widehat{h}_l(n)|\Big) ^2\Big) ^{1/2}\\
&\lesssim 2^{\frac{\rho}{2}l}\Big( \sum _{k\in \mathbb{Z}^3}2^{2\alpha j'}|\widehat{f}_{j'}(k)|^2\Big) ^{1/2}\Big( \sum _{k\in \mathbb{Z}^3}\sum _{\begin{smallmatrix} m,n\in \mathbb{Z}^3\\ (n,k,m)\in \La \end{smallmatrix}}2^{2\beta j}|\widehat{g}_j(m)|^2\cdot 2^{2\gamma l}|\widehat{h}_l(n)|^2\Big) ^{1/2}\\
&\lesssim 2^{\frac{\rho}{2}l}
\| f_{j'}\|_{H^\alpha}\| g_j\|_{H^\beta}\| h_l\|_{H^\gamma},
\end{align*}
so that
\[
S\lesssim \sum _{j\ge 0}\sum _{j'=j,j+1}\| f_{j'}\|_{H^\alpha}\| g_j\|_{H^\beta}\sum _{0\le l\le j}\| h_l\|_{H^\gamma}\cdot 2^{-\alpha j'-\beta j-\gamma l+\frac{\rho}{2}l}.
\]

When (i) holds, we define $\delta >0$ so that $\frac{\rho}{2}+\delta = \alpha +\beta +\gamma$.
By applying the Cauchy-Schwarz inequality in $j$ and $l$, we have
\eqq{\sum _{j\ge 0}\sum _{j'=j,j+1}\| f_{j'}\|_{H^\alpha}\| g_j\|_{H^\beta}\sum _{0\le l\le j}2^{-\de l}\| h_l\|_{H^\gamma}\lec \| f\|_{H^\alpha}\| g\|_{H^\beta}\| h\|_{H^\gamma}.}
It then suffices to show 
\[ p:=-\alpha j'-\beta j-\gamma l+\frac{\rho}{2}l+\delta l\le C\]
under the condition $0\le l\le j\le j'\le j+1$.

It is enough to consider the worst case that $\alpha \le \beta \le \gamma$.
By the assumption $\alpha +\beta \ge 0$, we have
\[ \alpha j'+\beta j+\gamma l=\alpha (j'-j)+(\alpha +\beta )j+\gamma l\ge -|\alpha |+(\alpha +\beta +\gamma )l,\]
which implies $p\le |\alpha |$.
This concludes the proof for the case (i).

When (ii) holds, we set $\de =\al +\be +\ga -\max \{ \al ,\be ,\ga \} >0$.
Since
\eqq{\sum _{j\ge 0}\sum _{j'=j,j+1}\| f_{j'}\|_{H^\alpha}\| g_j\|_{H^\beta}\sum _{0\le l\le j}2^{-\de (j-l)}\| h_l\|_{H^\gamma}\lec \| f\|_{H^\alpha}\| g\|_{H^\beta}\| h\|_{H^\gamma},}
it suffices to prove
\[ q:=-\alpha j'-\beta j-\gamma l+\frac{\rho}{2}l+\delta (j-l)\le C.\]
Considering the worst case $\alpha \le \beta \le \gamma$, we see that
\eqq{\alpha j'+\beta j+\gamma l=\alpha (j'-j) +(\al +\be )(j-l) +(\alpha +\be +\ga )l\ge -|\al |+\de (j-l) +\frac{\rho}{2}l,}
which implies $q\le |\al |$.
This concludes the proof for the case (ii).
\end{proof}

Next, we recall the following Sobolev estimate.
\begin{lem}\label{lem:Sobolev}
Let $d\ge 1$ and $\al,\be ,\ga \in \R$.
The inequality
\eqq{\big| \LR{fg,h}_{L^2(\T ^d)}\big| \lec \tnorm{f}{H^{\al}(\T ^d)}\tnorm{g}{H^{\be}(\T ^d)}\tnorm{h}{H^{\ga}(\T ^d)},}
or equivalently, 
\eqq{\tnorm{fg}{H^{-\ga}(\T ^d)}\lec \tnorm{f}{H^{\al}(\T ^d)}\tnorm{g}{H^{\be}(\T ^d)},}
holds if and only if one of the following is satisfied:
\begin{enumerate}
\item $\al +\be +\ga\ge \max \{ \al ,\be ,\ga \}$ and $\al +\be +\ga>\frac{d}{2}$;
\item $\al +\be +\ga> \max \{ \al ,\be ,\ga \}$ and $\al +\be +\ga=\frac{d}{2}$.
\end{enumerate}
\end{lem}

\begin{proof}
The `if' part follows from Lemma~\ref{lem:RC} with $\La =\shugo{k+m+n=0}$ and $\rho =d$.

To show the `only if' part, we assume $\al \le \be \le \ga$ without loss of generality.
It suffices to show that the quantity
\eqq{I(f,g,h):=\frac{\Big| \sum\limits _{\mat{k,n,m\in \Bo{Z}^d\\ k+m+n=0}}\hhat{f}(k)\hhat{g}(m)\hhat{h}(n)\Big|}{\tnorm{f}{H^\al}\tnorm{g}{H^\be}\tnorm{h}{H^{\ga}}}}
does not have an upper bound in each of the following cases:
(a) $\al +\be <0$, (b) $\al +\be +\ga <\frac{d}{2}$, (c) $\al +\be =0$ and $\ga =\frac{d}{2}$.

Let $N\gg 1$ be a positive integer and $e_1:=(1,0,\dots ,0)\in \Bo{Z}^d$.
In (a), we define
\eqq{\hhat{f}:=\chi _{\{ Ne_1\}},\quad \hhat{g}:=\chi _{\{ (-N-1)e_1\}},\quad \hhat{h}:=\chi _{\{ e_1\}},}
so that
\eqq{I(f,g,h)\sim \frac{1}{N^{\al}\cdot N^{\be}\cdot 1}\to \I \qquad (N\to \I ).}
In (b), we can use 
\eqq{\hhat{f}=\hhat{g}:=\chi _{\{ n\in \Bo{Z}^d\,|\,|n-2Ne_1|\le N\}},\quad \hhat{h}:=\chi _{\{ n\in \Bo{Z}^d\,|\,|n+4Ne_1|\le 2N\}},}
so that
\eqq{I(f,g,h)\gec \frac{N^{2d}}{N^{\al +\frac{d}{2}}\cdot N^{\be +\frac{d}{2}}\cdot N^{\ga +\frac{d}{2}}}=N^{\frac{d}{2}-(\al +\be +\ga )}\to \I \qquad (N\to \I ).}
In (c), we take
\eqs{\hhat{f}=\chi _{\{ n\in \Bo{Z}^d\,|\,|n-2Ne_1|\le N\}},\quad \hhat{g}:=\chi _{\{ n\in \Bo{Z}^d\,|\,|n+2Ne_1|\le N\}},\\
\hhat{h}:=|\cdot |^{-d}\chi _{\{ n\in \Bo{Z}^d\,|\,0<|n|\le N\}},}
so that
\eqq{I(f,g,h)\gec \frac{N^d\log N}{N^{\al +\frac{d}{2}}\cdot N^{\be +\frac{d}{2}}\cdot (\log N)^{\frac{1}{2}}}=(\log N)^{\frac{1}{2}}\to \I \qquad (N\to \I ),}
which finishes the proof.
\end{proof}

\subsection{Scaling invariance and optimality of the result}\label{Fourier-Lebesgue}

We first recall the scaling invariance of the fractional Navier-Stokes equations.
If $(u,p)$ is a solution of \eqref{NSC} with $\Om =0$, then $(u_\la ,p_\la )$ with
\eqq{u_\la (t,x):=\la ^{2\al -1}u(\la ^{2\al}t,\la x),\quad p_\la (t,x):=\la ^{4\al -2}p(\la ^{2\al}t,\la x),\qquad \la >0}
is also a solution with rescaled initial data
\eqq{u_{0,\la}(x):=\la ^{2\al -1}u_0(\la x)}
with $\mathrm{div}\,u_{0,\la}=0$.
Although such a rescaling changes the period of spatial domain, one can still consider the scaling critical regularity $s_c$ for which $\tnorm{u_{0,\la}}{\dot{H}^{s_c}(\mathbb{R}^d)}=\tnorm{u_{0}}{\dot{H}^{s_c}(\mathbb{R}^d)}$ for any $\la >0$.
We find that
\eqq{s_c=\frac{d}{2}+1-2\al .}

By the scaling heuristics, the local-in-time theory may be developed in $H^s$ for $s\ge s_c$ (sub-critical and critical regularities); equivalently, in $H^s$ with a fixed regularity $s$ for
\eq{condition1}{\al \ge \frac{d+2-2s}{4}.} 
For instance, local theory in $H^1$ requires $\al \ge \frac{1}{2}$ in 2D and $\al \ge \frac{3}{4}$ in 3D, as observed in the proof of Theorem~\ref{prop:lwp}.

Another regularity restriction may arise in the global a priori estimate for the limit equation \eqref{limit}.
For the 2D part $\bbar{U}(t)$, one can apparently gain one spatial derivative through the vorticity formulation; in fact, the nonlinear term $(\bbar{U}^h \cdot \nabla ^h)\om$ has the same scaling as square of $\om$ with no derivative.
Then, the 2D part may have an $H^s$ global a priori bound if the dissipation $\langle (-\Delta )^\al \om ,\om \rangle _{H^s}$ dominates the nonlinearity $\langle \om \om ,\om \rangle _{H^s}$ in the energy estimate.
Regarding an extra $w$ in the nonlinearity as extra $\frac{d}{2}$ derivatives (by the scaling heuristics $L^\I \sim H^{\frac{d}{2}}$) and compare the total number of derivatives in these terms, we find the condition
\eqq{2\al +2s>2s+\frac{d}{2}.}
Since $\bbar{U}$ behaves as a 2D flow, we set $d=2$ to come to the condition $\al >\frac{1}{2}$.

For the oscillating part $U_{\mathrm{osc}}$, the nonlinearity, which is quadratic with one derivative, has $1+\e$ dimensional interactions, as suggested in Lemma~\ref{lem:c2}.
We compare the number of derivatives just as above, but with $d=1+\e$, to see that the condition
\eq{condition2}{2\al +2s>2s+1+\frac{1+\e}{2}\quad \Leftrightarrow \quad \al >\frac{3}{4}}
is required for an $H^s$ global control on $U\osc$.

We remark that our result (Theorem~\ref{main}), global regularity in $H^1$ for $\al >\frac{3}{4}$, is optimal in both \eqref{condition1} and \eqref{condition2}.
In other words, by \eqref{condition1} one needs regularity $H^1$ to deal with $\al$ arbitrarily close to $\frac{3}{4}$; however, one may not relax the condition on $\al$ due to the restriction \eqref{condition2} even if the initial data is more regular than $H^1$. 
That is exactly why we work in $H^1$ in this article.

Hence, if we could prove Lemma~\ref{lem:c2} with just $CL^\e$ in the right-hand side, then the restriction on $\al$ would be relaxed to $\al >\frac{1}{2}$.
In this case, however, one has to work with higher regularity $H^{\frac{3}{2}}$, due to \eqref{condition1}.
Another natural space to work in is the Fourier-Lebesgue space $\F ^{-1}\ell ^1(\T ^3)$ defined in Remark~\ref{rem:Fourier-Lebesgue}.
This space has the same scaling as $H^{\frac{3}{2}}(\T ^3)$, while it is an algebra and continuously embedded into the space of (bounded uniformly) continuous functions on $\T ^3$.

\end{document}